\documentclass[11pt, reqno]{amsart}
\usepackage{amsmath,amsthm,amssymb,mathrsfs}
\usepackage[shortalphabetic,abbrev,non-sorted-cites]{amsrefs}
\usepackage{color}
\usepackage{latexsym}
\usepackage{datetime}
\usepackage{hyperref}

\theoremstyle{plain}
  \newtheorem{thm}{Theorem}[section]
  \newtheorem{lem}[thm]{Lemma}
  \newtheorem{prop}[thm]{Proposition}
  \newtheorem{cor}[thm]{Corollary} 
  \newtheorem*{thm*}{Theorem A}
  \newtheorem*{thma*}{Theorem C}
  \newtheorem*{thmb*}{Theorem D}
  \newtheorem*{thmc*}{Theorem B}
\theoremstyle{definition}
  
  \newtheorem{rmk}[thm]{Remark}
  \newtheorem*{ack*}{Acknowledgement}
  \newtheorem*{ques*}{Question}
\theoremstyle{plain}

\numberwithin{equation}{section}

\newcommand\pl{\partial}
\newcommand\dbr{\bar{\partial}}
\newcommand\dbrs{\bar{\partial}^*}

\newcommand\oh{\frac{1}{2}}
\newcommand\dd{{\mathrm d}}
\newcommand\bo{{\mathbf 1}}

\newcommand\CC{\mathcal{C}}
\newcommand\CD{\mathcal{D}}
\newcommand\CV{\mathcal{V}}
\newcommand\CE{\mathcal{E}}
\newcommand\CO{\mathcal{O}}

\newcommand\CK{\mathcal{K}}

\newcommand\CJ{\mathcal{J}}
\newcommand\CN{\mathcal{N}}

\newcommand\BC{\mathbb{C}}
\newcommand\BR{\mathbb{R}}
\newcommand\BN{\mathbb{N}}
\newcommand\BS{\mathbb{S}}
\newcommand\BI{\mathbb{I}}
\newcommand\BA{\mathbb{A}}
\newcommand\BF{\mathbb{F}}
\newcommand\BZ{\mathbb{Z}}

\newcommand\bx{\mathbf{x}}
\newcommand\rb{{\bf r}}

\newcommand\fe{\mathfrak{e}}

\newcommand\fb{\mathfrak{b}}
\newcommand\ff{\mathfrak{f}}
\newcommand\fg{\mathfrak{g}}
\newcommand\fh{\mathfrak{h}}
\newcommand\fD{\mathfrak{D}}

\newcommand\fL{\mathfrak{L}}
\newcommand\fR{\mathfrak{R}}

\newcommand\sr{\mathfrak{r}}
\newcommand\st{\mathfrak{t}}

\newcommand\mr{\mathring}
\newcommand\mb[1]{\overset{\centerdot}{#1}}
\newcommand\mbr{\breve}

\DeclareMathOperator{\sfa}{\mathsf f}
\DeclareMathOperator{\cl}{cl}
\DeclareMathOperator{\pr}{pr}
\DeclareMathOperator{\ev}{ev}
\DeclareMathOperator{\re}{Re}
\DeclareMathOperator{\spn}{span}
\DeclareMathOperator{\dist}{dist}
\DeclareMathOperator{\Hom}{Hom}
\DeclareMathOperator{\End}{End}
\DeclareMathOperator{\Dom}{Dom}

\begin{document}

\title[Dirac Spectral Flow on Contact $3$-Manifolds I]{Dirac Spectral Flow on Contact Three Manifolds I:\\
Eigensection Estimates and Spectral Asymmetry}

\author[C.-J. Tsai]{Chung-Jun Tsai}
\address{Department of Mathematics\\ and National Center for Theoretical Sciences (Mathematics Division, Taipei Office)\\
National Taiwan University\\ Taipei 10617\\ Taiwan}
\email{cjtsai@ntu.edu.tw}

\date{\usdate{\today}}

\maketitle

\begin{abstract}
Let $Y$ be a compact, oriented $3$-manifold with a contact form $a$ and a metric $\dd s^2$.  Suppose that ${ F}\to Y$ is a principal bundle with structure group ${\rm U}(2) = {\rm SU}(2)\times_{\scriptscriptstyle\{\pm1\}}S^1$ such that ${ F}/S^1$ is the principal ${\rm SO}(3)$ bundle of orthonormal frames for $TY$.  A unitary connection $A_0$ on the Hermitian line bundle ${ F}\times_{\scriptscriptstyle\det{\rm U}(2)}\BC$ determines a self-adjoint Dirac operator $\CD_0$ on the $\BC^2$-bundle ${ F}\times_{\scriptscriptstyle{\rm U}(2)}\BC^2$.

The contact form $a$ can be used to perturb the connection $A_0$ by $A_0 - ira$.  This associates a one parameter family of Dirac operators $\CD_r$ for $r\geq0$.  When $r>\!>1$, we establish a sharp sup-norm estimate on the eigensections of $\CD_r$ with small eigenvalues.  
The sup-norm estimate can be applied to study the asymptotic behavior of the spectral flow from $\CD_0$ to $\CD_r$.  In particular, it implies that the subleading order term of the spectral flow is strictly smaller than ${\CO}(r^{\frac{3}{2}})$.  We also relate the $\eta$-invariant of $\CD_r$ to certain spectral asymmetry function involving only the small eigenvalues of $\CD_r$.
\end{abstract}


\section{Introduction}

In Taubes's proof of the Weinstein conjecture \cite{ref_Taubes_SW_Weinstein}, a key ingredient is the spectral flow estimate for a one parameter family of Dirac operators.  The spectral flow estimate has a natural generalization \cite{ref_Taubes_sf} to any odd dimensional manifolds.  Although being used to prove the Weinstein conjecture, the spectral flow estimate is established in a general setting.  When the one parameter family of Dirac operators is constructed from a contact form, it is interesting to see how its spectral flow function and the zero eigensections are related to the geometry of the contact form.  This paper is the first step toward the study of this question.

\subsection{Spin-c Dirac operator in three dimension}\label{subsec_spinc}
Suppose that $Y$ is a compact, oriented $3$-manifold with a Riemannian metric $\dd s^2$.  Let ${ Fr}$ be the principal ${\rm SO}(3)$ bundle of oriented, orthonormal frames.  A \emph{spin-c structure} on $Y$ is an equivalent class of lifting of ${ Fr}$ to a principal ${\rm Spin}^\BC(3) = {\rm U}(2)$ bundle.  In dimension three, the spin-c structures can be constructed explicitly.  Since any compact oriented $3$-manifold is parallelizable, ${ Fr}$ can be identified with $Y\times{\rm SO}(3)$.  It suggests an obvious spin-c structure, the trivial ${\rm U}(2)$ bundle ${ F} = Y\times{\rm U}(2)$.  Let $U\to Y$ be a principal $S^1$ bundle.  The principal bundle $F\times_{S^1}U$ is a different spin-c structure if $U$ is non-trivial, where $S^1$ acts on ${\rm U}(2)$ as its center.  This construction identifies the set of spin-c structures on $Y$ with the set of equivalent classes of principal $S^1$ bundles.  Note that the equivalent classes of $S^1$ bundles is an affine space isomorphic to ${\rm H}^2(Y;\BZ)$.

Let $\BS$ be the associated bundle of $F$ by the fundamental representation of ${\rm U}(2)$ on $\BC^2$.  It is called the \emph{spinor bundle}.  The Levi-Civita connection on $Fr$ and a unitary connection $A$ on $\det(\BS) = U\times_{S^1}\BC$ together induce a unitary connection on $\BS$.  Denote the connection by $\nabla_{A}$.

The tangent bundle $TY$ admits an action on $\BS$ defined as follows.  Identify $\BR^3$ with $2\times2$ skew Hermitian matrices.  The group ${\rm U}(2)$ acts on $\BR^3$ by ${\bf x}\mapsto g{\bf x}g^*$ for any ${\bf x}\in\BR^3$ and $g\in{\rm U}(2)$.  The associated bundle of $F$ of this representation is exactly the tangent bundle $TY$.  The matrix action of a $2\times2$ skew Hermitian matrix on $\BC^2$ induces a bundle map $\cl:TY\times\BS\to\BS$.  This map is called the Clifford action.  The Dirac operator $D_A$ associated to $\nabla_{A}$ is the composition of the following maps
\begin{align*}
\CC^\infty(Y;\BS) \stackrel{\nabla_{A}}{\longrightarrow} \CC^\infty(Y;T^*Y\otimes\BS) \stackrel{\text{metric dual}}{\longrightarrow} \CC^\infty(Y;TY\otimes\BS) \stackrel{\cl}{\longrightarrow} \CC^\infty(Y;\BS) ~.
\end{align*}
The Dirac operator is self-adjoint with respect to the $L^2$-inner product.  It has discrete spectrum and each eigenvalue has finite multiplicity.  Moreover, its eigenvalues is unbounded from above  and below.

There are two different conventions for the Clifford action.  The convention in this paper is determined by what follows:  suppose that $\{e_1, e_2, e_3\}$ is an oriented, orthonormal basis of tangent vectors, then $\cl(e_1)\cl(e_2) = -\cl(e_3)$.

\subsection{Dirac spectral flow}
Suppose that $\BS$ is a spinor bundle.  Let $A_0$ and $A_1$ be unitary connections on $\det(\BS)$.  Choose a path of unitary connections $\{A_s\}_{s\in[0,1]}$ on $\det(\BS)$ which starts at $A_0$ and ends at $A_1$.  This path associates a path of Dirac operators from $D_{A_0}$ to $D_{A_1}$.  The \emph{Dirac spectral flow} is the algebraic count of the zero crossings of eigenvalues: a zero crossing contributes to the count with $+1$ if the eigenvalue crosses zero from a negative to a positive value as $s$ increases, and count with $-1$ if the eigenvalue crosses zero from a positive to a negative value as $s$ increases.  For a generic choice of the path $\{A_s\}_{s\in[0,1]}$, only these two cases occur.  This algebraic count is the Dirac spectral flow.  The complete definition of the spectral flow can be found in \cite[\S7]{ref_APS3} and \cite[\S5.1]{ref_Taubes_SW_Weinstein}.

Atiyah, Patodi and Singer \cite[p.95]{ref_APS3} observed that the spectral flow function is equal to the index of certain Dirac operator on $[0,1]\times Y$ with appropriate boundary conditions.  They also proved that this index \cite[(4.3)]{ref_APS1} is path independent \cite[p.89]{ref_APS3}.  Therefore, the spectral flow function depends only on the ordered pair $({A_0},{A_1})$, but not on the path $\{A_s\}_{s\in[0,1]}$.

\smallskip

Given a real-valued $1$-form $a$, we can consider the spectral flow from $A_0$ to $A_0 - {ir}a$ for any $r\geq1$.  The spectral flow can be thought as a function of $r$, which we denote by $\sfa_a(A_0,r)$.  In \cite[\S5]{ref_Taubes_SW_Weinstein} and \cite{ref_Taubes_sf}, Taubes studied the asymptotic behavior of $\sfa_a(A_0,r)$ as $r\to\infty$.  He proved:

\begin{thm*} (\cite[Proposition 5.5]{ref_Taubes_SW_Weinstein})
There exist a universal constant $\delta\in(0,\oh)$ and a constant $c_1$ determined by $\dd s^2$ and $A_0$ such that
\begin{align*}
\big|\sfa_a(A_0,r)-\frac{r^2}{32\pi^2}\int_Y a\wedge\dd a\big| \leq c_1r^{\frac{3}{2}+\delta}
\end{align*}
for any real-valued $1$-form $a$ with $||a||_{\mathcal{C}^3}\leq 1$ and any $r\geq c_1$.
\end{thm*}
This theorem specifies the leading order term of the spectral flow function, and gives a bound on the subleading order term.

\subsection{Spectral flow on contact three manifolds}
A \emph{contact form} $a$ on an oriented three manifold is a $1$-form such that $a\wedge\dd a>0$.  An \emph{adapted metric} on a contact three manifold is a Riemannian metric so that $|a| = 1$ and $\dd a = 2*a$, where $*$ is the Hodge star operator.  Chern and Hamilton \cite{ref_CH} proved that such a metric always exists.

Suppose that $(Y,a)$ is a contact three manifold, and $\dd s^2$ is an adapted metric.  Suppose that $D_{A_0}$ is a spin-c Dirac operator on $Y$.  The zero eigensections of the Dirac operator $D_{A_0-ira}$ have the following properties when $r>\!>1$.
\begin{itemize}
\item The \emph{Reeb vector field} is the unique vector field $v$ such that $\dd a(v,\cdot) = 0$ and $a(v) = 1$.  The covariant derivative of the zero eigensection along $v$ is close to the multiplication by $ir/2$.  Thus, its magnitude does not change much along the Reeb vector field $v$.
\item The \emph{contact hyperplane} (or the \emph{contact structure}) is the two dimensional distribution in $TY$ defined by $\ker(a)$.  On the contact hyperplanes, the zero eigensections almost satisfy a Cauchy--Riemann equation.
\end{itemize}
The precise statements will appear in \S\ref{sec_pt_est}.  These properties suggest that instead of the Riemannian geometry in three dimension, the scenery here is more like the complex geometry in one dimension.  It motivates the following questions.
\begin{ques*}
Suppose that $(Y,a)$ is a contact three manifold with an adapted metric $\dd s^2$.
\begin{enumerate}
\item Is the subleading order term of $\sfa_a(A_0,r)$ of order $r$ instead of order $r^{\frac{3}{2}+\delta}$?  If this being the case, what is the coefficient of the subleading order term, and  what is its geometric meaning?
\item What is the relation between the zero locus of the zero eigensection of $D_{A_0-ira}$ and the behavior of the Reeb vector field as $r\to\infty$?
\end{enumerate}
\end{ques*}

\subsection{Main results}
The main result of this paper is that the subleading order term of the spectral flow function is of ${\scriptscriptstyle \CO}(r^{\frac{3}{2}})$.  It sort of suggests that the answer to Question (i) is affirmative.

\begin{thmc*} [{Theorem \ref{cor_sf_402}(ii)}]
Suppose that $(Y,a)$ is a contact three manifold with an adapted metric $\dd s^2$.  Suppose that $D_{A_0}$ is a spin-c Dirac operator.  Then, there exists a constant $c_2$ determined by $a$, $\dd s^2$ and $A_0$ such that
\begin{align*}
\big| \sfa_a(A_0,r) - \frac{r^2}{32\pi^2}\int_Y a\wedge\dd a \big| &\leq
c_2 r^{\frac{3}{2}}(\log r)^{-\oh} \qquad\text{ for any } r\geq c_2 ~.
\end{align*}
\end{thmc*}

There are two main ingredients in the proof of Theorem B.  The following theorem is the first ingredient.  It investigates the eigensections of $D_{A_0-ira}$ with small eigenvalues.

\begin{thma*} [Theorem \ref{thm_point_estimate_00}]
Suppose that $(Y,a)$ is a contact three manifold with an adapted metric $\dd s^2$.  Suppose that $D_{A_0}$ is a spin-c Dirac operator.  For any positive $r$ and $\lambda$, let
$$ \CV(r,\lambda) = \spn\big\{\psi\in\CC^\infty(Y;\BS) ~\big|~ D_{A_0-ira}\psi=\nu\psi, \text{ for some scalar } \nu\text{ with }|\nu|\leq\lambda \big\} ~.$$
Then, there exists a constant $c_3$ determined by $a$, $\dd s^2$ and $A_0$ such that
\begin{align*}
\sup_Y |\psi|^2 &\leq c_3r\lambda\int_Y|\psi|^2
\end{align*}
for any $r\geq c_3$, $1\leq\lambda\leq\oh r^\oh$ and $\psi\in\CV(r,\lambda)$.
\end{thma*}

This theorem implies (Corollary \ref{cor_main_thm_00}(i)) that
\begin{align}\label{eqn_dim_est}   \dim\CV(r,\lambda)\leq c_3r\lambda ~.   \end{align}
It provides another evidence that $D_{A_0-ira}$ behaves more like the complex geometry in one dimension.  If there is no condition on the $1$-form $a$, it is very likely that \cite[Proposition 2.2]{ref_Taubes_sf} is the best statement one can make.  With the help of the heat kernel argument, this dimension estimate (\ref{eqn_dim_est}) leads to the following estimate on the spectral flow function.  It is the second ingredient in the proof of Theorem B.

\begin{thmb*} [{Theorem \ref{cor_sf_402}(i)}]
Suppose that $(Y,a)$ is a contact three manifold with an adapted metric $\dd s^2$.  Suppose that $D_{A_0}$ is a spin-c Dirac operator.  Then, there exists a constant $c_4$ determined by $a$, $\dd s^2$ and $A_0$ such that
\begin{align*}
\big| \sfa_a(A_0,r) - \frac{r^2}{32\pi^2}\int_Y a\wedge\dd a - \dot{\eta}(A_0-ira) \big| &\leq
c_4 r(\log r)^{\frac{9}{2}} 
\end{align*}
for any $r\geq c_4$.  The function $\dot{\eta}(A_0-ira)$ is defined by
$$ \big(\frac{80}{\pi}\big)^\oh r^{-\oh}(\log r)^\oh \Big(\sum_{\psi\in\CV_{r}^+}\int^{\frac{1}{3}r^\oh}_{\lambda_\psi} e^{-20(r^{-1}\log r)u^2}\,\dd u - \sum_{\psi\in\CV_{r}^-}\int_{-\frac{1}{3}r^\oh}^{\lambda_\psi} e^{-20(r^{-1}\log r)u^2}\,\dd u\Big)$$
where $\CV_{r}^+$ consists of orthonormal eigensetions of $D_{A_0-ira}$ whose eigenvalue belongs to $(0,\frac{1}{3}r^\oh)$, $\CV_{r}^-$ consists of orthonormal eigensetions of $D_{A_0-ira}$ whose eigenvalue belongs to $(-\frac{1}{3}r^\oh,0)$, and $\lambda_\psi$ is the corresponding eigenvalue of $\psi$.  (The constants $\frac{1}{3}$ and $20$ are not crucial.  They are just convenient choices.)
\end{thmb*}

Theorem D says that we only need to focus on the small eigenvalues of $D_{A_0-ira}$ in order to study the spectral flow from $A_0$ to $A_0-ira$.  With the help of (\ref{eqn_dim_est}), both summations of $\dot{\eta}(A_0-ira)$ can be shown to be smaller than $c_5 r^{\frac{3}{2}}(\log r)^{-\oh}$.  That is to say,
\begin{align*}  \big(\frac{80}{\pi}\big)^\oh r^{-\oh}(\log r)^\oh \Big( \big|\sum_{\psi\in\CV_{r}^+}\int^{\frac{1}{3}r^\oh}_{\lambda_\psi} e^{-20(r^{-1}\log r)u^2}\,\dd u\big| \qquad & \\
 + \big|\sum_{\psi\in\CV_{r}^-}\int_{-\frac{1}{3}r^\oh}^{\lambda_\psi} e^{-20(r^{-1}\log r)u^2}\,\dd u\big| \Big) &\leq 2c_5 r^{\frac{3}{2}}(\log r)^{-\oh} ~,   \end{align*}
and Theorem B follows.

If the eigenvalues of $\CV_r^+\cup\CV_r^-$ are `uniformly distributed', one can image that $\dot{\eta}(A_0-ira)$ is actually much smaller than $r^{\frac{3}{2}}(\log r)^{-\oh}$ due to cancellation.  In the sequel of this paper \cite{ref_Ts3}, the `uniformly distributed' property will be justified for certain types of contact forms in each isotopy class of contact structures.

\subsection{Spectral asymmetry}
By combining with  the results of Atiyah, Patodi and Singer, Theorem D has an interesting corollary.  As a background for the corollary, consider the four manifold $X = [0,r]\times Y$.  The spinor bundle $\BS\to Y$ can naturally be regarded as a bundle over $X$.  Define the operator $\fD:\CC^\infty(X;\BS)\to\CC^\infty(X;\BS)$ by
$$ \fD = \frac{\pl}{\pl s} + D_{A_0-isa} $$
where $s$ is the parameter for the interval $[0,r]$.  With appropriate boundary conditions (\cite[(2.3)]{ref_APS1}), the operator $\fD$ is a Fredholm operator from $L^2_1(X,\BS)\to L^2(X,\BS)$.  As observed by \cite[p.95]{ref_APS3}, the index of $\fD$ is equal to the spectral flow from $A_0$ to $A_0 - ira$.  Meanwhile, \cite[(4.3) and pp.59--60]{ref_APS1} gives a formula for the index of $\fD$.  Their result in the present setting says that
\begin{align}\label{eqn_eta_01}\begin{split}
\sfa_a(A_0,r) &= \frac{r^2}{32\pi^2}\int_Ya\wedge\dd a + \frac{r}{16\pi^2}\int_Ya\wedge({i}F_{A_0}) \\
&\qquad + \frac{1}{2}(h({A_0-ira}) + \eta({A_0-ira}) - h({A_0}) - \eta({A_0}))
\end{split}\end{align}
where  $h(A)$ is the dimension of $\ker(D_A)$ and $\eta(A)$ is the spectral asymmetry function of $D_A$.  This spectral asymmetry function is defined as follows:  it is the value at $z=0$ of the analytic continuation to $\BC$ of
\begin{align*}
\sum_{\psi}{\rm sign}(\lambda_\psi)|\lambda_\psi|^{-z} \qquad\text{ defined on where } \re(z)>\!>1 ~.
\end{align*}
The summation is indexed by an orthonormal eigenbasis of $D_A$ with nonzero eigenvalue, and $\lambda_{\psi}$ is the eigenvalue of $\psi$.  Theorem 3.10 of \cite{ref_APS1} asserts that the analytic continuation is finite at $z=0$.  One can also see \cite[\S1]{ref_Nicolaescu_eta} for a nice survey on the $\eta$-invariant and the formula (\ref{eqn_eta_01}).

Roughly speaking, $\eta(A)$ measures the difference between the total number of positive eigenvalues and the total number of negative eigenvalues.  As pointed out by Taubes \cite[Corollary 3]{ref_Taubes_sf}, Theorem A and formula (\ref{eqn_eta_01}) imply that the subleading order term of the spectral flow function is the same as
$$ \oh\big(h(A_0-ira) + \eta(A_0-ira)\big)  $$
up to an $\CO(r)$ difference.

Let $(Y,a)$ be a contact three manifold with an adapted metric $\dd s^2$.  The dimension estimate (\ref{eqn_dim_est}) implies that
$$   h({A_0 - ira})\leq c_3r ~.   $$
It follows from Theorem D and (\ref{eqn_eta_01}) that there exists an $r$-independent constant $c_6$ such that
\begin{align}
\big|\eta({A_0-ira}) - 2\dot{\eta}(A_0-ira)\big| &\leq c_6r(\log r)^{\frac{9}{2}}
\end{align}
for any $r\geq c_6$.  This relates the full spectral asymmetry to the spectral asymmetry involving only small eigenvalues.  It would be interesting if one can say something about the behavior of $\eta(A_0-ira)$ as $r\to\infty$ without using the spectral flow.

\begin{rmk}
The constants $c_{(\cdot)}$ in this paper are always \emph{independent} of $r$.  In other words, they only depend on the contact form $a$, the metric $\dd s^2$ and the connection $A_0$.  The subscript is simply to indicate that these constants might increase/decrease after each step.  The subscript will be returned to $1$ at the beginning of each section.
\end{rmk}

\subsection{Contents of this paper}
This paper is divided into three parts.

\S\ref{sec_basic_estimate} and \S\ref{sec_pt_est} are devoted to the proof of Theorem C.  The Clifford action of the contact form on $\BS$ is skew-Hermitian.  It induces the eigenbundle splitting $\BS = E_1\oplus E_2$, where $\cl(a)$ acts as $i|a|$ and $-i|a|$, respectively.  With respect to this splitting,  a section $\psi\in\CC^\infty(Y;\BS)$ can be written as $(\alpha,\beta)$.  There are three observations based on this splitting.  The first observation is that $\beta$ is much smaller than $\alpha$.  Secondly, on a small disk transverse to the Reeb vector field, the $E_2$-component of the Dirac equation reads
$$ (\pl_x + i\pl_y)(\alpha) = \text{smaller terms such as }\beta $$
where $x$ and $y$ are local coordinate on the disk.  Lastly, the $E_1$-component of the Dirac equation implies that the integral of $|\alpha|^2$ over a transverse disk is bounded by its integral over $Y$.  That is to say, the integral of $|\alpha|^2$ do not concentrate on some particular disk.  With this understood, the strategy is to estimate the sup-norm of $\beta$ and other smaller terms by the sup-norm of $\alpha$.  Then apply the Cauchy integral formula to estimate the sup-norm of $\alpha$.

In \S\ref{sec_heat_kernel} we apply the parametrix technique to study the heat kernel for the square of the Dirac operator $D_{A_0-ira}$.  With an \emph{a priori} estimate on the heat kernel, the parametrix argument generates a small time expansion of the heat kernel.  The accuracy of the output relies on the original \emph{a priori} estimate.  Proposition \ref{prop_heat_00} supplies such an \emph{a priori} estimate.  It uses Theorem C to obtain a $L^2$ estimate (in space) of the heat kernel.  

In \S\ref{sec_sf} we discuss on the spectral flow from $D_{A_0}$ to $D_{A_0-ira}$.  Let $\CE_r$ be the following eigenvalue configuration:
$$  \CE_r = \big\{(s,\lambda) ~\big|~ 0<s<r,~\lambda\in{\rm spec}(D_{A_0-isa}),\text{ and }|\lambda|<\frac{1}{3}r^\oh\big\} ~.  $$
We assign a displacement function $\Psi$ to $\CE_r$.  The displacement $\Psi(\CE_r)$ is closely related to the spectral flow $\sfa_a(A_0,r)$.  Its behavior for $r>\!>1$ can be computed by the heat kernel expansion.  The main purpose of \S\ref{sec_sf} is to prove Theorem B and Theorem D by this displacement $\Psi(\CE_r)$.

\begin{ack*}
The author would like to thank Cliff Taubes for the support and for helpful comments on an earlier draft of this paper.
\end{ack*}

\section{Dirac Operator on Contact Three Manifolds}\label{sec_basic_estimate}
Suppose that $(Y,a)$ is a contact three manifold.  A metric $\dd s^2$ is called \emph{conformally adapted} if $\dd s^2 = \Omega^2\dd\mr{s}^2$ for some adapted metric $\dd\mr{s}^2$ and some smooth function $\Omega$ with
\begin{align*}  \frac{9}{10}\leq\Omega\leq\frac{10}{9}  ~. \end{align*}
The function $\Omega$ is called the \emph{conformal factor}.  The particular bounds chosen here are just convenient normalizations; any other fixed bounds would do the job.  This notion is a minor generalization of an adapted metric.  It is designed to handle some technical issue in \cite[\S4]{ref_Ts3}.

\subsection{Spectral flow and conformal change of the metric}\label{subsec_conformal}
Many spectral properties of a Dirac operator are invariant under conformal changes of metric.  The main purpose of this subsection is to review some of them.  Denote by $D_A$ the associated Dirac operator using the metric $\dd s^2$, and by $\mr{D}_A$ the associated Dirac operator using the metric $\dd\mr{s}^2$.

In \cite[\S1.4]{ref_Hitchin} Hitchin found the transformation formula between $D_A$ and $\mr{D}_A$, which is explained as follows.  The spinor bundles using $\dd s^2$ and $\dd\mr{s}^2$ can be thought as the {same} bundle with the {same} Hermitian metric.  With this understood, the Clifford actions of $TY$ are related by
\begin{align}\label{eqn_conformal_01}
\cl(u) &= \Omega\,\mr{\cl}(u)
\end{align}
for any tangent vector $u$.  The Dirac operators are related by
\begin{align}\label{eqn_conformal_02}
D_A\psi &= \Omega^{-\frac{n+1}{2}}\mr{D}_A(\Omega^{\frac{n-1}{2}}\psi) = \Omega^{-2} \mr{D}_A(\Omega\psi)
\end{align}
for any $\psi\in\CC^\infty(Y;\BS)$.  The formula in the middle wors for any dimension $n$.  It follows that the dimension of $\ker(D_A)$ is a conformal invariant (\cite[Proposition 1.3]{ref_Hitchin}).

In \cite{ref_Hitchin}, Hitchin did the computation for the trivial spin-c structure (or the spin structure).  Since $\cl(w) = \Omega^{-1}\mr{\cl}(w)$ for any $1$-form $w$, the formula (\ref{eqn_conformal_02}) holds for any spin-c Dirac operator as well.  It can be seen from the local expression of the Dirac operator (\cite[(3.3)]{ref_Morgan}).

\smallskip
Besides the dimension of the kernel, the spectral flow function $\sfa_a(A_0,r)$ is also a conformal invariant.  A na\"ive reason is that the spectral flow is constructed by counting the dimension of the kernel of associated Dirac operators.  

According to (\ref{eqn_eta_01}), the conformal invariance of the spectral flow function $\sfa_a(A_0,r)$ follows from the conformal invariance of the $\eta$-invariant.  The latter property is proved by Atiyah, Patodi and Singer \cite[pp.420--421]{ref_APS2} for certain Dirac operator, and by Rosenberg \cite[Theorem 3.8]{ref_Rosenberg} for general Dirac operators.

\subsection{Canonical spin-c structure of a contact form}\label{subsec_can_spin}
As described in \cite[\S2.1]{ref_Taubes_SW_Weinstein}, the spin-c structures and spin-c Dirac operators can be seen more geometrically with the help of the contact form.  Suppose that $\dd s^2 = \Omega^2\dd\mr{s}^2$ is a conformally adapted metric.

Since the Reeb vector field $v$ is nowhere vanishing, it induces the splitting $\BS = E_1\oplus E_2$ of any spinor bundle into eigenbundles for $\cl(v)$.  The convention here is that $\cl(v)$ acts as $i|v|$ on $E_1$ and as $-i|v|$ on $E_2$.  There is a \emph{canonical spin-c structure} determined by the contact form $a$, that where the bundle $E_1$ is the trivial bundle.  The splitting of the \emph{canonical spinor bundle} is written as $\underline{\BC}\oplus K^{-1}$, where $K^{-1}$ is isomorphic as an $\rm{SO}(2)$ bundle to $\ker(a)$ with the orientation given by $\dd a$.  To be more precise, let $J$ be the rotation counterclockwisely on $\ker(a)$ by $90$ degree.  The rotation operator $J$ is determined by $\dd s^2$ and $\dd a$.  The local sections of $K^{-1}$ consists of $u-i J(u)$ for any $u\in\ker(a)$.

\smallskip
The conformally adapted metric determines a \emph{canonical connection} on the canonical spinor bundle $\underline{\BC}\oplus K^{-1}$.  Let $\bo$ be the unit-normed, trivializing section of $\underline{\BC}$.  The canonical connection is the unique spin-c connection such that the associated Dirac operator annihilates the section $\Omega^{-1}\bo$.  The proof for its existence and uniqueness can be found in \cite[Lemma 10.1]{ref_Hutchings}.

\begin{rmk}
The Dirac operator of the canonical connection satisfies the transformation rule (\ref{eqn_conformal_02}).  Let $\underline{\BC}\oplus \mr{K}^{-1}$ be the canonical spinor bundle using $\dd\mr{s}^2$.  The metrics $\dd s^2$ and $\dd\mr{s}^2$ define the same rotation operator $J$.  It follows that the isometric identification of the canonical spinors bundles is characterized by
\begin{align*}\begin{array}{ccc}
\underline{\BC}\oplus K^{-1}  &\longrightarrow  &\underline{\BC}\oplus\mr{K}^{-1} \\
(\bo, u-i J(u))  &\mapsto  &\big(\bo, \Omega(u-i J(u))\big)
\end{array} ~. \end{align*}
Since the canonical connection is \emph{uniquely} determined by the annihilation property, the canonical connections of $\dd s^2$ must become the canonical connection of $\dd\mr{s}^2$ under the above identification.
\end{rmk}

\smallskip
Any two spin-c structures differ by the tensor product with a complex line bundle \cite[Appendix D]{ref_LM}.  The specification of a canonical spin-c structure allows us to write any spinor bundle as
$$  \BS = E\oplus EK^{-1}  $$
for some Hermitian line bundle $E\to Y$.  Its determinant bundle $\det(\BS)$ is $E^2K^{-1}$.  Let $A_{\text{can}}$ be the connection on $K^{-1} = \det(\underline{\BC}\oplus K^{-1})$ that induces the canonical connection.  Any connection on $E^2K^{-1}$ can be written as $A_0 = A_{\text{can}} + 2A_E$ for some unitary connection $A_E$ on $E$.  In other words, a unitary connection $A_E$ on $E$ determines a unitary connection $A_0$ on $\det(\BS)$, and hence determines a spin-c connection on $\BS = E\oplus EK^{-1}$.

\smallskip
We abbreviate $D_{A_0-ira}$ as $D_r$, and the spectral flow function $\sfa_a(A_0,r)$ as $\sfa_a(r)$.  The above settings and notations (the contact form, conformally adapted metric and spin-c Dirac operators) will be used \emph{throughout the rest of this paper}.

\subsection{Some basic estimates}

With the splitting $\BS = E\oplus EK^{-1}$, the following proposition provides a fundamental estimate on components of the eigensections of $D_r$.

\begin{prop}\label{prop_beta_estimate_00}
There exists a constant $c_1$ determined by the contact form $a$, the conformally adapted metric $\dd s^2$ and the connection $A_0$ on $\det(\BS)$ such that the following holds.
\begin{enumerate}
\item For any $r\geq c_1$, suppose that $\psi$ is a eigensection of $D_r$ such that $|\lambda_\psi|^2< \frac{3}{4} r$.  Then
\begin{align*}
\int_Y |\beta|^2 + r^{-1}\int_Y|\nabla_r\beta|^2 \leq c_1r^{-1} \int_Y |\alpha|^2
\end{align*}
where $\alpha$ is the $E$ component of $\psi$, and $\beta$ is the $EK^{-1}$ component of $\psi$.
\item Suppose that there is a continuous path of eigenvalues $\lambda(s)$ of $D_{s}$ which is smooth at $r\geq c_1$ and  $|\lambda(r)|^2<\frac{3}{4} r$.  Then
\begin{align*}
\frac{9}{20} - c_1r^{-1} \leq \lambda'(r) \leq \frac{5}{9} ~.
\end{align*}
In particular, there are only \emph{positive} zero crossings for the spectral flow of $D_s$ when $s\geq 3c_1$.
\end{enumerate}
\end{prop}
\begin{proof}
(\emph{Assertion} (i))\;  The proof is essentially the same as that of Proposition 3.1(i) in \cite{ref_Ts1}.  The key is the Weitzenb\"ock formula:
\begin{align}\label{eqn_weitzenbock_00}D_r^2\psi = \nabla_r^*\nabla_r\psi + \frac{\kappa}{4}\psi + \cl(\frac{F_{A_0}}{2})\psi - ir\cl(\frac{\dd a}{2})\psi
\end{align}
where $\kappa$ is the scalar curvature.  Since $*\dd a = 2\Omega^{-1}a$ with respect to the metric $\dd s^2$, the Clifford action $\cl(\dd a/2)$ is equal to $-\Omega^{-1}\cl(a)$. Pair (\ref{eqn_weitzenbock_00}) with $\beta$, and integrate over $Y$.  After integration by parts, we find that
\begin{align*}
\lambda_\psi^2\int_Y|\beta|^2 &\geq \int_Y \big( (\frac{81}{100}r - c_2)|\beta|^2 + \oh|\nabla_r\beta|^2 - c_2|\alpha|^2 \big)
\end{align*}
for some constant $c_2$.  Assertion (i) of the proposition follows from this inequality.

\smallskip
(\emph{Assertion} (ii))\;  According to \cite[\S5.1]{ref_Taubes_SW_Weinstein}, there exists a constant $\epsilon_1>0$ such that the multiplicity of $\lambda(s)$ of $D_{s}$ is a constant for any $s\in(r,r+\epsilon_1)$, and $\lambda(s)$ is smooth when $s\in(r,r+\epsilon_1)$.  Due to \cite[(5.4)]{ref_Taubes_SW_Weinstein}, the derivative of $\lambda(s)$ is given by 
\begin{align}\label{eqn_basic_est_01}
\lambda'(s) = \int_Y\langle\psi_s,-\frac{i}{2}\cl(a)\psi_s\rangle = \int_Y\oh\Omega^{-1}\big(|\alpha_s|^2 - |\beta_s|^2\big)
\end{align}
where $\psi_s=(\alpha_s,\beta_s)$ is a unit-normed eigensection of $D_{s}$ with eigenvalue $\lambda(s)$.  Since $|\lambda(r)|^2<\frac{3}{4}r$, there exists some positive constant $\epsilon_2\leq\epsilon_1$ such that $|\lambda(s)|^2<\frac{3}{4}r$ for any $s\in(r,r+\epsilon_2)$.  It follows from Assertion (i) and (\ref{eqn_basic_est_01}) that
\begin{align*}
\frac{9}{20} - c_3r^{-1} \leq \lambda'(s) \leq \frac{5}{9}
\end{align*}
for any $s\in(r,r+\epsilon_2)$.  Since $\lambda'(s) = \lim_{s\to r^+}\lambda'(s)$, it completes the proof of the proposition.
\end{proof}

As a remark, (\ref{eqn_basic_est_01}) implies that
\begin{align}\label{eqn_basic_est_02}
|\lambda'| &\leq \frac{5}{9}
\end{align}
without any assumption on $\lambda$.

\section{Pointwise Estimate on Eigensections}\label{sec_pt_est}
Let $\CV(r,\lambda)$ be the vector space spanned by eigensections of $D_r$ whose eigenvalue has magnitude less than or equal to $\lambda$.  Namely,
\begin{align*}
\CV(r,\lambda) = \spn\big\{ \psi\in\CC^\infty(Y;\BS) ~\big|~ D_r\psi = \nu\psi, \text{ for some scalar }\nu\text{ with }|\nu|\leq\lambda \big\} ~.
\end{align*}
This main purpose of this section is to prove the following pointwise estimate on $\psi\in\CV(r,\lambda)$.

\begin{thm}\label{thm_point_estimate_00}
There exists a constant $c_1$ determined by the contact form $a$, the conformally adapted metric $\dd s^2$ and connection $A_0$ on $\det(\BS)$ such that the following holds.  Suppose that $r\geq c_1$ and $1\leq\lambda\leq \oh r^\oh$, then
\begin{align}
\sup_Y |\psi|^2 &\leq c_1r\lambda \int_Y|\psi|^2
\end{align}
for any $\psi\in\CV(r,\lambda)$.
\end{thm}

Notice that Proposition \ref{prop_beta_estimate_00}(i) only holds for an \emph{individual} eigensection.  A generic element in $\CV(r,\lambda)$ is a linear combination of eigensections.  What follows is a modified version.

\begin{lem}\label{lem_CV_beta}
There exists a constant $c_2$ determined by the contact form $a$, the conformally adapted metric $\dd s^2$ and the connection $A_0$ such that:  for any $r\geq c_2$ and $1\leq\lambda\leq\oh r^\oh$,
\begin{align*}
\int_Y |\beta|^2 + r^{-1}\int_Y|\nabla_r\beta|^2 \leq c_2r^{-1}\lambda^2 \int_Y |\alpha|^2
\end{align*}
for any $\psi=(\alpha,\beta)\in\CV(r,\lambda)\subset\CC^\infty(Y;E\oplus EK^{-1})$.
\end{lem}

\begin{proof}
For any $k\in\BN$, consider the $k$th power of the Dirac operator $D_r$.  If $\psi$ belongs to $\CV(r,\lambda)$, $D_r^k\psi$ also belongs to $\mathcal{V}(r,\lambda)$ for any $k\in\BN$.  By writing $\psi$ as a linear combination of $L^2$-orthonormal eigenbases, it is not hard to see that
\begin{align}\label{eqn_estimate_00}
\int_Y |D_r^k\psi|^2 \leq \lambda^{2k} \int_ Y |\psi|^2 ~.
\end{align}
In particular, $\int_Y|D_r^2\psi|^2\leq\lambda^4\int_Y|\psi|^2$ for any $\psi\in\CV(r,\lambda)$.  With the same computation as that in the proof of Proposition \ref{prop_beta_estimate_00}(i),
\begin{align*}
\int_Y \big( (\frac{81}{100}r - c_3)|\beta|^2 + \oh|\nabla_r\beta|^2 - c_3|\alpha|^2 \big) &\leq \int_Y|D_r^2\psi||\beta| \\
&\leq \lambda^2\big(\int_Y(|\alpha|^2+|\beta|^2)\big)^\oh\,\big(\int_Y|\beta|^2\big)^\oh \\
&\leq 1000\lambda^2\int_Y|\alpha^2| + \frac{11}{10}\lambda^2\int_Y|\beta|^2 ~,
\end{align*}
and the lemma follows.
\end{proof}

\subsection{Corollaries of the sup-norm estimate}
Before getting into the proof, here are some useful consequences of Theorem \ref{thm_point_estimate_00}.
\begin{cor}\label{cor_main_thm_00}
There exists a constant $c_1$ determined by the contact form $a$, the conformally adapted metric $\dd s^2$ and the connection $A_0$ with the following significance.
\begin{enumerate}
\item Suppose that $r\geq c_1$ and $1\leq\lambda\leq \oh r^\oh$.  Let $\{\psi_j\}_{j\in J}$ be an orthonormal eigenbasis for $\mathcal{V}(r,\lambda)$.  Then,
\begin{align*}
\sum_{j\in J} |\psi_j(q)|^2 &\leq c_1r\lambda
\end{align*}
for any $q\in Y$.  Its integration over $Y$ says that $\dim \mathcal{V}(r,\lambda) \leq c_1r\lambda$.
\item For any $r\geq c_1$, the spectral flow from $r-1$ to $r$ is less than or equal to $c_1r$. Namely, $\sfa_a(r)-\sfa_a(r-1)\leq c_1r$.
\end{enumerate}
\end{cor}

\begin{proof}
(\emph{Assertion} (i))\;	For any $q\in Y$, choose isometric identifications $E|_q\cong\BC$ and $EK^{-1}|_q\cong\BC$.  With these identifications, write $\psi_j(q) = (\alpha_j(q),\beta_j(q))\in\BC^2$, and introduce the following linear maps on $L^2(Y;\BS)$
\begin{align*} \begin{array}{cccl}
&L^2(Y;\BS) &\to &\mathbb{C} \\
\ev_q^1: &\psi &\mapsto &\int_Y\langle\psi(p),\sum_{j\in J}\bar{\alpha}_j(q)\psi_j(p)\rangle\dd p ~; \\
\ev_q^2: &\psi &\mapsto &\int_Y\langle\psi(p),\sum_{j\in J}\bar{\beta}_j(q)\psi_j(p)\rangle\dd p ~.
\end{array}\end{align*}
It is a standard fact in functional analysis that $\ev_q^1$ and $\ev_q^2$ are bounded linear functionals, and the operator norms are equal to $(\sum_{j\in J}|\alpha_j(p)|^2)^\oh$ and $(\sum_{j\in J}|\beta_j(p)|^2)^\oh$, respectively.

Let $\Pi_\lambda:L^2(Y;\BS)\to\CV(r,\lambda)$ be the $L^2$-orthogonal projection.  For any $\psi\in\CC^\infty(Y;\BS)$, the linear functionals are equal to
\begin{align*}
\ev_q^1(\psi) = (\pr_1\circ\ev_q\circ\Pi_\lambda)(\psi) \qquad\text{and}\qquad
\ev_q^2(\psi) = (\pr_2\circ\ev_q\circ\Pi_\lambda)(\psi)
\end{align*}
where $\ev_q$ is the evaluation map at $q$, $\pr_1$ is the projection onto the $E$ component, and $\pr_2$ is the projection onto the $EK^{-1}$ component.  According to Theorem \ref{thm_point_estimate_00},
\begin{align*}
\big| \ev_q^1(\psi) \big|^2 + \big| \ev_q^2(\psi) \big|^2 &\leq \sup_Y |\Pi_\lambda(\psi)|^2 \leq c_1r\lambda \int_Y|\Pi_\lambda(\psi)|^2 \\
&\leq c_1r\lambda \int_Y |\psi|^2 ~.
\end{align*}
It follows that the operator norm of $\ev_q^1$ and $\ev_q^2$ are no greater than $(c_1r\lambda)^\oh$.  This completes the proof of Assertion (i).

\medskip

(\emph{Assertion} (ii))\;	Suppose that $\{r_k\}_{k=1}^K$ are where the zero crossing happens between $r-1$ and $r$ (counting multiplicities).  According to \cite[\S5.1]{ref_Taubes_SW_Weinstein}, one can assign for each $k$ a continuous, piecewise smooth function $\lambda_k(s)$ of $s\in[r-1,r]$ such that
\begin{itemize}
\item $\lambda_k(s)$ is an eigenvalue $D_s$ for $s\in[r-1,r]$, and $\lambda_k(r_k) = 0$;
\item moreover, $\{\lambda_k(s)\}_{k=1}^K$ are disjoint eigenvalues (counting multiplicities) of $D_s$ for any $s\in[r-1,r]$.
\end{itemize}
There is no canonical way to do it, but any method will suffice.  It follows from (\ref{eqn_basic_est_02}) that $\lambda_k(s)$ always belongs to $(-1,1)$ for $s\in[r-1,r]$.  Thus, $K<\dim\CV(r,1)$, which is less than $c_1r$ by Assertion (i).
\end{proof}

When the metric is adapted rather than conformally adapted, the dimension estimate of Corollary \ref{cor_main_thm_00}(i) can be refined into a density version.  For an adapted metric, the slope estimate of Proposition \ref{prop_beta_estimate_00}(ii) is refined to be
\begin{align}\label{eqn_slope_01}
|\lambda'(r)-\oh|\leq c_2r^{-1}
\end{align}
provided $\lambda(r)$ is an eigenvalue of $D_r$ with $|\lambda(r)|^2\leq\frac{3}{4}r$.  This is proved in \cite[Proposition 3.1(ii)]{ref_Ts1}.  Notice that the leading order term of the slope is exactly $\oh$.

\begin{cor}\label{cor_eigen_distribution}
Suppose that the metric is adapted, namely $\Omega\equiv1$.  There exists a constant $c_3$ determined by the contact form $a$, the conformally adapted metric $\dd s^2$ and the connection $A_0$ such that the following holds.  Suppose that $r\geq c_3$ and $\lambda_-,\lambda_+\in[\oh r^\oh,\oh r^\oh]$ satisfying $0<\lambda_+-\lambda_-\leq2$.  Then, the total number of eigenvalues (counting multiplicity) of $D_r$ within $[\lambda_-,\lambda_+]$ is no greater than $c_3r$.
\end{cor}

\begin{proof}
Consider the case when $\lambda_-\geq0$.  Other cases can be proved by the same argument.  Suppose that $\lambda_-=\lambda_1\leq\lambda_2\leq\cdots\leq\lambda_L=\lambda_+$ are all the eigenvalues of $D_r$ within $[\lambda_-,\lambda_+]$.  For each $l\in\{1,2,\cdots,L\}$, assign a continuous, piecewise smooth function $\lambda_l(s)$ for $s\in[r-\frac{21}{10}\lambda_+,r]$ such that
\begin{itemize}
\item $\lambda_l(s)$ is an eigenvalue of $D_s$ for $s\in[r-\frac{21}{10}\lambda_+,r]$, and $\lambda_l(r) = \lambda_l$;
\item moreover, $\{\lambda_l(s)\}_{l=1}^L$ are disjoint eigenvalues (counting multiplicities) of $D_s$ for any $s\in[r-\frac{21}{10}\lambda_+,r]$.
\end{itemize}
There is no canonical way to do it, but any method will suffice.

We claim that $|\lambda_l(s)|^2<\frac{3}{4}s$ for any $s\in[r-\frac{21}{10}\lambda_+,r]$ and any $l\in\{1,2,\cdots,L\}$.  Due to (\ref{eqn_basic_est_01}), $|\lambda_l(s)-\lambda_l|\leq\frac{1}{2}(r-s)$ for any $s\in[r-\frac{21}{10}\lambda_+,r]$.  It follows that $|\lambda_l(s)|\leq\frac{31}{20}\lambda_+$, and
$$ |\lambda_l(s)|^2\leq\frac{961}{1600}r\leq\frac{3}{4}(r-3\lambda_+)\leq\frac{3}{4}s ~.$$
Hence, (\ref{eqn_slope_01}) applies to $\lambda_l(s)$.  According to the intermediate value theorem, there is some
$$  s_l\in[r-(\oh+c_4r^{-1})\lambda_l,r-(\oh-c_4r^{-1})\lambda_l]  $$ such that $\lambda_l(s_l) = 0$.  It follows that
$$  L \leq \sfa_a(r-(\oh-c_4r^{-1})\lambda_-) - \sfa_a(r-(\oh+c_4r^{-1})\lambda_+) ~.  $$  Since $|\lambda_\pm|\leq\frac{1}{3}r^\oh$ and $\lambda_+ - \lambda_-\leq2$, the corollary follows from Corollary \ref{cor_main_thm_00}(ii).
\end{proof}

\subsection{Pointwise estimate on $\beta$}
The rest of this section is devoted to the proof of Theorem \ref{thm_point_estimate_00}.  Suppose that $\psi=(\alpha,\beta)$ is an element of $\CV(r,\lambda)$ for some $\lambda\leq\oh r^\oh$.  Proposition \ref{prop_beta_estimate_00}(i) says that the $L^2$-norm of $\beta$ is small.  The purpose of this subsection is to derive a pointwise estimate on $\beta$.

The following lemma is a preliminary version of Theorem \ref{thm_point_estimate_00}.

\begin{lem}\label{lem_estimate_00}
There exists a constant $c_6$ determined by the contact form $a$, the conformally adapted metric $\dd s^2$ and the connection $A_0$ such that the following holds.  Suppose that $r\geq c_6$ and $\lambda\leq\oh r^\oh$, then
\begin{align*}  \sup_Y |\psi|^2 \leq c_6r^{\frac{3}{2}}\int_Y |\psi|^2  \end{align*}
for any $\psi\in\mathcal{V}(r,\lambda)$.  On the other hand, if $\lambda\geq\oh r^\oh$, then
\begin{align*}  \sup_Y |\psi|^2 \leq c_6\lambda^{3}\int_Y |\psi|^2 \end{align*}
for any $\psi\in\mathcal{V}(r,\lambda)$.
\end{lem}
\begin{proof}
Suppose that the maximum of $|\psi|$ is achieved at $p_0\in Y$.  Let $\chi$ be a standard cut-off function which depends only on the distance $\rho$ to $p_0$ and
\begin{align*} \begin{cases}
\chi(\rho) = 1 &\text{when }\rho\leq\epsilon_1 ~, \\
\chi(\rho) = 0 &\text{when }\rho\geq2\epsilon_1 ~.
\end{cases} \end{align*}
Here, $\epsilon_1$ is a small number less than one-tenth of the injectivity radius, and the precise value will be chosen later.  Due to the Weitzenb\"ock formula (\ref{eqn_weitzenbock_00}), $\chi\psi$ satisfies the following differential inequality:
\begin{align*}
\dd^*\dd |\chi\psi|^2 &\leq \chi^2\dd^*\dd|\psi|^2 + 8\chi|\dd\chi||\psi||\nabla_r\psi| + |\dd^*\dd(\chi^2)||\psi|^2 \\
&\leq \big(2\chi^2\langle\nabla_r^*\nabla_r\psi,\psi\rangle - 2\chi^2|\nabla_r\psi|^2\big) + \big(2\chi^2|\nabla_r\psi|^2 + 8|\dd\chi|^2|\psi|^2\big) \\
&\quad + 2(\chi|\dd^*\dd\chi| + |\dd\chi|^2)|\psi|^2 \\
&\leq c_7r|\chi\psi|^2 + c_7(\chi|\dd^*\dd\chi|+|\dd\chi|^2)|\psi|^2 + 2|\chi\psi||\chi D_r^2\psi| ~.
\end{align*}
Let $B$ be the geodesic ball centered at $p_0$ with radius to be half of the injectivity radius.  The cut-off function $\chi$ vanishes on $\pl B$.  By the maximum principle, $|(\chi\psi)(p_0)|^2$ is less than the Green's function of $\dd^*\dd$ acting on the right-hand side.  Since the three dimensional Green's function is bounded from above by $c_8\rho^{-1}$,
\begin{align*}
|\psi(p_0)|^2 &\leq c_9r\int_B \rho^{-1}|\chi\psi|^2 + c_9\epsilon_1^{-3}\int_B |\psi|^2 + c_9\int_B \rho^{-1}|\chi\psi||\chi D_r^2\psi| \\
&\leq c_9r\int_B \rho^{-1}|\chi\psi|^2 + c_9\epsilon_1^{-3}\int_B |\psi|^2 + c_9\epsilon_2^{-1}\int_B \rho^{-2}|\chi\psi|^2 + c_9\epsilon_2\int_B |D_r^2\psi|^2
\end{align*}
for any $\epsilon_2>0$.  Since $\sup|\psi| = |\psi(p_0)|$ ,the first term can be estimated in terms of $\psi(p_0)$:
\begin{align*}
\int_B \rho^{-1}|\chi\psi|^2 &\leq |\psi(p_0)|^2\int_{\dist(\cdot,p_0)\leq\epsilon_1}\rho^{-1} + \epsilon_1^{-1}\int_{\dist(\cdot,p_0)\geq\epsilon_1}|\chi\psi|^2 \\
&\leq c_{10}\epsilon_1^2 |\psi(p_0)|^2 + c_{10}\epsilon_1^{-1} \int_B|\psi|^2 ~.
\end{align*}
By the same token, the third term is less than or equal to
\begin{align*}
\int_B \rho^{-2}|\chi\psi|^2 \leq c_{11}\epsilon_1 |\psi(p_0)|^2 + c_{11}\epsilon_1^{-2} \int_B|\psi|^2 ~.
\end{align*}
The above inequalities together with (\ref{eqn_estimate_00}) for $k=2$ imply that
\begin{align*}
|\psi(p_0)|^2 &\leq c_{12}(r\epsilon_1^2 + \epsilon_1\epsilon_2^{-1})|\psi(p_0)|^2 + c_{12}(r\epsilon_1^{-1} + \epsilon_1^{-3} + \epsilon_1^{-2}\epsilon_2^{-1} + \epsilon_2\lambda^4)\int_Y |\psi|^2 ~.
\end{align*}
By taking $\epsilon_1 = (100c_{12}r)^{-\oh}$ and $\epsilon_2 = c_{12}^{\oh}r^{-\oh}$, the first assertion of the lemma follows.  For the second assertion, take $\epsilon_1 = (1000c_{12})^{-\oh}\lambda^{-1}$ and $\epsilon_2 = c_{12}^{\oh}\lambda^{-1}$.
\end{proof}

Since $D_r^k\psi$ still belongs to $\CV(r,\lambda)$ for any $\psi\in\CV(r,\lambda)$, Lemma \ref{lem_estimate_00} applies to $D_r^k\psi$ as well.
\begin{cor}\label{cor_estimate_00}
There exists a constant $c_{13}$ determined by the contact form $a$, the conformally adapted metric $\dd s^2$ and the connection $A_0$ with the following significance.  Suppose that $r\geq c_{13}$ and $\lambda\leq \oh r^\oh$, then
\begin{align*}
\sup_Y |D_r\psi|^2 &\leq c_{13}r^{\frac{3}{2}}\lambda^2\int_Y|\psi|^2
&\text{and}& &\sup_Y |D_r^2\psi|^2 &\leq c_{13}r^{\frac{3}{2}}\lambda^4\int_Y|\psi|^2
\end{align*} for any $\psi\in\CV(r,\lambda)$.
\end{cor}
\begin{proof}
It follows from Lemma \ref{lem_estimate_00} and (\ref{eqn_estimate_00}).
\end{proof}

The second assertion of Lemma \ref{lem_estimate_00} implies the following dimension bound of $\CV(r,\lambda)$ for $\lambda\geq \oh r^\oh$.
\begin{cor}\label{cor_estimate_01}
There exists a constant $c_6$ determined by the contact form $a$, the conformally adapted metric $\dd s^2$ and the connection $A_0$ with the following property.  Suppose that $r\geq c_6$ and $\lambda\geq \oh r^\oh$. Let $\{\psi_j\}_{j\in J}$ be an orthonormal eigenbasis for $\mathcal{V}(r,\lambda)$.  Then $\sum_{j\in J} |\psi_j(p)|^2 \leq c_6\lambda^{3}$ for any $p\in Y$.  It follows that $\dim\CV(r,\lambda)\leq c_6\lambda^{3}$.
\end{cor}
\begin{proof}
This corollary follows from the same functional analysis argument as that for Corollary \ref{cor_main_thm_00}(i).
\end{proof}

The following proposition gives a pointwise estimate on $\beta$ in terms of $\alpha$.

\begin{prop}\label{prop_estimate_00}
There exists a constant $c_{15}$ determined by the contact form $a$, the conformally adapted metric $\dd s^2$ and the connection $A_0$ such that the following holds.  For $r\geq c_{15}$ and $1\leq\lambda\leq\oh r^\oh$, suppose that $\psi=(\alpha,\beta)$ is an element in $\mathcal{V}(r,\lambda)$.  Then,
\begin{align*} \sup_Y|\beta|^2 &\leq c_{15}r^{-1}\sup_Y|\alpha|^2 + c_{15}r^{-\oh}\lambda^4\int_Y|\psi|^2 ~. \end{align*}
It follows that $\sup_Y |\psi|^2 \leq (1+c_{15}r^{-1})\sup_Y|\alpha|^2 + c_{15}r^{-\oh}\lambda^4\int_Y|\psi|^2$.
\end{prop}
\begin{proof}
Project the Weitzenb\"{o}ck formula (\ref{eqn_weitzenbock_00}) onto the summand of $E$ and $EK^{-1}$, and 
take the inner product with $\alpha$ and $\beta$, respectively.  It leads to the following inequalities:
\begin{align*}
\oh\dd^*\dd|\alpha|^2 + |\nabla_r\alpha|^2 - \frac{100}{81}r|\alpha|^2 &\leq c_{16}\big( |\alpha|^2 + |\beta||\alpha| + |\nabla_r\beta||\alpha| + |D_r^2\psi||\alpha| \big) ~, \\
\oh\dd^*\dd|\beta|^2 + |\nabla_r\beta|^2 + \frac{81}{100} r|\beta|^2 &\leq c_{16}\big( |\beta|^2 + |\alpha||\beta| + |\nabla_r\alpha||\beta| + |D_r^2\psi||\beta| \big) ~.
\end{align*}
Due to Corollary \ref{cor_estimate_00} and the Cauchy--Schwarz inequality, they become:
\begin{align*}
&\dd^*\dd|\alpha|^2 + 2|\nabla_r\alpha|^2 \leq c_{17}\big( r|\alpha|^2 + r^{-1}|\beta|^2 + r^{-1}|\nabla_r\beta|^2 + r^\oh\lambda^4\int_Y|\psi|^2 \big) ~, \\
&\dd^*\dd|\beta|^2 + 2|\nabla_r\beta|^2 + r|\beta|^2 \leq c_{17}\big( r^{-1}|\alpha|^2 + r^{-1}|\nabla_r\alpha|^2 + r^\oh\lambda^4\int_Y|\psi|^2 \big) ~.
\end{align*}
It follows that the combination $|\beta|^2 + c_{17}r^{-1}|\alpha|^2$ obeys the following differential inequality:
\begin{align}\label{eqn_estimate_06}\begin{split}
& \dd^*\dd(|\beta|^2 + c_{17}r^{-1}|\alpha|^2) + r(|\beta|^2+ c_{17}r^{-1}|\alpha|^2) + (|\nabla_r\beta|^2 + c_{17}r^{-1}|\nabla\alpha|^2) \\
\leq\, & c_{18}|\alpha|^2 + c_{18}r^{\oh}\lambda^4\int_Y|\psi|^2 ~.
\end{split}\end{align}
Let $\zeta$ be the function
$$ \zeta \equiv |\beta|^2 + c_{17}r^{-1}|\alpha|^2 - c_{18}r^{-1}\sup_Y|\alpha|^2 - c_{18}r^{-\oh}\lambda^4\int_Y|\psi|^2 ~. $$
The equation (\ref{eqn_estimate_06}) implies that $\dd^*\dd\zeta + r\,\zeta\leq 0$.  By the maximum principle, $\zeta$ cannot have positive maximum.  This finishes the proof of the proposition.
\end{proof}

\subsection{Pointwise estimate on covariant derivatives}
To prove Theorem \ref{thm_point_estimate_00}, some estimate on the covariant derivative of $\psi$ is needed.  The following lemma provides a preliminary estimate on $\nabla_r\psi$.

\begin{lem}\label{lem_estimate_01}
There exists a constant $c_{20}$ determined by the contact form $a$, the conformally adapted metric $\dd s^2$ and the connection $A_0$ with the following significance.  For any $r\geq c_{20}$ and $1\leq\lambda\leq\oh r^\oh$, suppose that $\psi\in\mathcal{V}(r,\lambda)$.  Then,
\begin{align*} \sup_Y|\nabla_r\psi|^2 &\leq c_{20}r^{\frac{5}{2}}\int_Y|\psi|^2  \end{align*}
and $\sup_Y|\nabla_r(D_r^2\psi)|^2 \leq c_{20}r^{\frac{5}{2}}\lambda^4\int_Y|\psi|^2$.
\end{lem}
\begin{proof}
The first step is to estimate the $L^2$-norm of $\nabla_r\psi$.  Integrating the Weitzenb\"{o}ck formula (\ref{eqn_weitzenbock_00}) against $\psi$ implies that $\int_Y |\nabla_r\psi|^2 \leq c_{21}r\int_Y|\psi|^2 + \int_Y|D_r^2\psi||\psi| $.  It follows from (\ref{eqn_estimate_00}) and the Cauchy--Schwarz inequality that
\begin{align}\label{eqn_estimate_01}\begin{split}
\int_Y |\nabla_r\psi|^2 &\leq c_{22}r\int_Y |\psi|^2 \quad\text{ and } \\
\int_Y |\nabla_r(D_r^2\psi)|^2 &\leq c_{22}r\int_Y |D_r^2\psi|^2 \leq c_{23}r\lambda^4 \int_Y |\psi|^2 ~.
\end{split}\end{align}

Commuting covariant derivatives gives the following formulae:
\begin{align}\label{eqn_estimate_02}
\nabla_r^*\nabla_r\nabla_r\psi - \nabla_r\nabla_r^*\nabla_r\psi
&= ir\,\dd a(\nabla_r\psi,\cdot) - \oh ir(\dd^*\dd a)\otimes\psi + {Q}_1(\nabla_r\psi) + {Q}_2(\psi)
\end{align}
where $Q_1$ and $Q_2$ are operators defined from the contact form $a$, the metric $\dd s^2$ and the connection $A_0$; in particular, neither depends on $r$, and neither is a differential operator.  The computation for (\ref{eqn_estimate_02}) is included in \S\ref{apx_commuting}.  The significance of (\ref{eqn_estimate_02}) is that the crucial terms of right hand side are $r\nabla_r\psi$ and $r\psi$.

The term $\nabla_r\nabla_r^*\nabla_r\psi$ can be replaced by the covariant derivative of (\ref{eqn_weitzenbock_00}).  Let $\chi$ be a cut-off function.  After some simple manipulations, $\chi\nabla_r\psi$ obeys the following differential inequality:
\begin{align*}
\dd^*\dd|\chi\nabla_r\psi|^2 &\leq c_{24}r|\chi\nabla_r\psi|^2 + c_{24}|\chi\nabla_r\psi|\big(r|\chi\psi| + |\chi\nabla_r(D_r^2\psi)|\big) \\
&\quad + c_{24}(\chi|\dd^*\dd\chi| + |\dd\chi|^2)|\nabla_r\psi|^2 ~.
\end{align*}
The same Green's function argument as that in the proof of Lemma \ref{lem_estimate_00} shows that
\begin{align*}
\sup_Y |\nabla_r\psi|^2 &\leq c_{25}\big( r^{\frac{3}{2}}\int_Y|\nabla_r\psi|^2 + r^{-\oh}\int_Y|\nabla_r(D_r^2\psi)|^2 + r^{\frac{3}{2}}\int_Y|\psi|^2 \big) ~.
\end{align*}
This estimate and (\ref{eqn_estimate_01}) together prove the first assertion.  The second assertion follows from the first assertion and (\ref{eqn_estimate_00}).  This completes the proof of the lemma.
\end{proof}

The following lemma provides a refined estimate on $\nabla_r\psi$.
\begin{lem}
There exists a constant $c_{26}$ determined by the contact form $a$, the conformally adapted metric $\dd s^2$ and the connection $A_0$ such that the following holds.  For any $r\geq c_{26}$ and $1\leq\lambda\leq \oh r^\oh$, suppose that $\psi\in\mathcal{V}(r,\lambda)$.  Then
\begin{align*} \sup_Y|\nabla_r\psi|^2 &\leq c_{26}r\sup_Y|\psi|^2 + c_{26}r^{\oh}\lambda^4\int_Y|\psi|^2 ~. \end{align*}
\end{lem}
\begin{proof}
Take the inner product of (\ref{eqn_weitzenbock_00}) with $\psi$ and apply Corollary \ref{cor_estimate_00} to obtain the following differential inequality:
\begin{align}\label{eqn_estimate_08}\begin{split}
\oh\dd^*\dd|\psi|^2 + |\nabla_r\psi|^2 &\leq 2r|\psi|^2 + |\psi||D_r^2\psi| \\
&\leq c_{27}r |\psi|^2 + c_{27}r^{\oh}\lambda^4\int_Y|\psi|^2 ~.
\end{split}\end{align}
Similarly, take the inner product of (\ref{eqn_estimate_02}) with $\nabla_r\psi$, replace $\nabla_r\nabla_r^*\nabla_r\psi$ by the covariant derivative of (\ref{eqn_weitzenbock_00}), and apply Lemma \ref{lem_estimate_01} to obtain the following differential inequality:
\begin{align}\label{eqn_estimate_05}\begin{split}
\oh\dd^*\dd|\nabla_r\psi|^2 + |\nabla_r\nabla_r\psi|^2 &\leq c_{28}r|\nabla_r\psi|^2 + c_{28}r|\nabla_r\psi||\psi| + |\nabla_r\psi||\nabla_r(D_r^2\psi)| \\
&\leq c_{29}r(|\nabla_r\psi|^2 + |\psi|^2) + c_{29}r^{\frac{3}{2}}\lambda^4\int_Y|\psi|^2 ~.
\end{split}\end{align}
It follows from (\ref{eqn_estimate_08}) and (\ref{eqn_estimate_05}) that
\begin{align*}
\dd^*\dd(|\nabla_r\psi|^2 + c_{30}r|\psi|^2) + r(|\nabla_r\psi|^2 + c_{30}r|\psi|^2) &\leq c_{31}r^2|\psi|^2 + c_{31}r^{\frac{3}{2}}\lambda^4\int_Y|\psi|^2 ~.
\end{align*}
By the maximum principle,
$$  |\nabla_r\psi|^2 + c_{30}r|\psi|^2 - c_{31}r\sup_Y|\psi|^2 - c_{31}r^\oh\lambda^4\int_Y|\psi|^2  $$ cannot admit positive maximum.  This completes the proof of the lemma.
\end{proof}

Lemma \ref{lem_CV_beta} says that the $L^2$-norm of $\nabla_r\beta$ cannot be not large.  The following proposition is a pointwise version.

\begin{prop}\label{prop_estimate_02}
There exists a constant $c_{35}$ determined by the contact form $a$, the conformally adapted metric $\dd s^2$ and the connection $A_0$ such that the following holds.  For any $r\geq c_{35}$ and $1\leq\lambda\leq \oh r^\oh$, suppose that $\psi\in\mathcal{V}(r,\lambda)$.  Then
\begin{align*} \sup_Y|\nabla_r\beta|^2 &\leq c_{35}\sup_Y|\alpha|^2 + c_{35}r^{\oh}\lambda^4\int_Y|\psi|^2 ~. \end{align*}
\end{prop}
\begin{proof}
In order to derive the equation for $\nabla_r\beta$, consider (\ref{eqn_estimate_02}) for $\beta$:
\begin{align}\label{eqn_estimate_03}
\nabla_r^*\nabla_r\nabla_r\beta = \nabla_r\nabla_r^*\nabla_r\beta + ir\dd a(\nabla_r\beta,\cdot) - \oh ir(\dd^*\dd a)\otimes\beta + {Q}_1(\nabla_r\beta) + {Q}_2(\beta) .
\end{align}
The connection Laplacian on $\beta$ can be formally expressed in terms of $\psi$:
\begin{align*}
\nabla_r^*\nabla_r\beta &= \pr_2(\nabla_r^*\nabla_r\psi) + Q_3(\nabla_r\psi) + Q_4(\psi) \\
&= \pr_2(D_r^2\psi) - r\Omega^{-2}\beta + Q_3(\nabla_r\psi) + Q_5(\psi) ~.
\end{align*}
The first equality is a straightforward computation, and the second equality follows from (\ref{eqn_weitzenbock_00}).  Here, $Q_3$, $Q_4$ and $Q_5$ are operators defined from the contact form, the metric and the base connection; in particular, none depends on $r$, and none is a differential operator.  The covariant derivative of the above equation reads
\begin{align}\label{eqn_estimate_04}\begin{split}
\nabla_r\nabla_r^*\nabla_r\beta &= -r\nabla_r(\Omega^{-2}\beta) + Q_6(\nabla_r(D_r^2\psi)) + Q_7(D_r^2\psi) \\
&\quad + Q_8(\nabla_r\nabla_r\psi) + Q_9(\nabla_r\psi) + Q_{10}(\psi)
\end{split}\end{align}
where all the $Q_j$ are independent of $r$, and they are not differential operators.

Take the inner product of (\ref{eqn_estimate_03}) with $\nabla_r\beta$, and substitute $\nabla_r\nabla_r^*\nabla_r\beta$ by (\ref{eqn_estimate_04}).  After applying the Cauchy--Schwarz inequality, it becomes the following differential inequality:
\begin{align*}
\oh\dd^*\dd|\nabla_r\beta|^2 &\leq c_{36}\big( r|\nabla_r\beta|^2 + r|\beta|^2 + r^{-1}|\nabla_r(D_r^2\psi)|^2 + r^{-1}|D_r^2\psi|^2 \\
&\quad\qquad + r^{-1}|\nabla_r\nabla_r\psi|^2 + r^{-1}|\nabla_r\psi|^2 + r^{-1}|\psi|^2\big) ~.
\end{align*}
To proceed, apply Lemma \ref{lem_estimate_01} on $|\nabla_r(D_r^2\psi)|^2$ and Corollary \ref{cor_estimate_00} on $|D_r^2\psi|^2$.  Then add (\ref{eqn_estimate_05}) multiplied by $c_{36}r^{-1}$ to cancel $c_{36}r^{-1}|\nabla_r\nabla_r\psi|^2$.  It ends up with the following inequality:
\begin{align*}
\oh\dd^*\dd \zeta_1 + r\,\zeta_1 &\leq c_{37}\big(r|\nabla_r\beta|^2 + |\nabla_r\psi|^2 + |\psi|^2 + r^{\frac{3}{2}}\lambda^4\int_Y|\psi|^2\big)
\end{align*}
where
$$   \zeta_1 = |\nabla_r\beta|^2 + c_{36}r^{-1}|\nabla_r\psi|^2 ~.   $$
The first three terms on the right hand side can be canceled by adding (\ref{eqn_estimate_06}) multiplied by $c_{37}r$.  It leads to the following inequality:
\begin{align*}
\dd^*\dd \zeta_2 + c_{38}r\,\zeta_2 &\leq c_{39}\big( r|\alpha|^2 + r^{\frac{3}{2}}\lambda^4\int_Y |\psi|^2 \big)
\end{align*}
where
$$   \zeta_2 = \zeta_1 + c_{37}r(|\beta_2| + c_{17}r^{-1}|\alpha|^2) ~.   $$  The maximum principle implies that
$   \zeta_2 - c_{40}\sup_Y|\alpha|^2 - c_{40}r^\oh\lambda^4\int_Y|\psi|^2   $
cannot have positive maximum for some constant $c_{40}$.  This completes the proof of the proposition.
\end{proof}

\subsection{Estimate the integral over a transverse disk}
The purpose of this subsection is to estimate the integral of $\alpha$ over a transverse disk.  This is a local computation.  It is easier to work with the adapted metric $\dd\mr{s}^2$ instead of the conformally adapted metric $\dd s^2 = \Omega^2\dd\mr{s}^2$.  Let $\mr{\psi} = (\mr{\alpha},\mr{\beta})$ be $\Omega\psi = (\Omega\alpha, \Omega\beta)$.  Note that $\mr{\psi}$ and $\psi$ have uniformly equivalent sup-norms and $L^2$-norms.  According to (\ref{eqn_conformal_02}), the equation for $\mr{\psi}$ reads
\begin{align}\label{eqn_conformal_03}
\mr{D}_r\mr{\psi} &= \Omega^2 D_r\psi ~.
\end{align}

\subsubsection{Adapted coordinate chart}\label{subsec_adapted_chart}
Given an adapted metric $\dd\mr{s}^2$, \cite[\S6.4]{ref_Taubes_SW_Weinstein} introduces the notion of an \emph{adapted coordinate chart}.  For any $p\in Y$, the {adapted coordinate chart} centered at $p$ is defined as the follows.  Denote by $v$ the Reeb vector field.  Choose two oriented, orthonormal vectors $e_1$ and $e_2$ for $\ker(a)|_p$.  For any $\ell>0$, let $I_\ell$ be the interval $[-\ell,\ell]$, and $C_\ell$ be the standard disk of radius $\ell$ in $\mathbb{R}^2$.  Consider
\begin{align*}  \begin{array}{cccl}
&C_\ell\times I_\ell &\to & Y \\
\varphi_0: &((x,y),0) &\mapsto & \exp_{p}(xe_1+ye_2) ~, \\
\varphi: &((x,y),z) &\mapsto & \exp_{\varphi_0(x,y)}(zv)
\end{array}  \end{align*}
where $\exp$ is the geodesic exponential map of $\dd\mr{s}^2$.  The map $\varphi$ defines a smooth embedding for sufficiently small $\ell$.  Similar to the injectivity radius,  the constant
\begin{align*}
\ell_a = \oh\inf_{p\in Y} \big( \sup \{ \ell>0 ~|~ \varphi\text{ defines a smooth embedding on } C_\ell\times I_\ell \text{ centered at }p \} \big)
\end{align*}
is strictly positive, and depends only on the contact form $a$ and the adapted metric $\dd\mr{s}^2$.  For any $p\in Y$, the adapted coordinate chart at $p$ is $\varphi(C_{\ell_a}\times I_{\ell_a})$.  For simplicity, the subscript $\ell_a$ will be suppressed.  The adapted coordinate chart has the following properties.
\begin{enumerate}
\item The Reeb vector field $v$ is $\pl_z$, and $\dd a = 2 B\,\dd x\wedge\dd y$.  The function $B$ is positive, and independent function of $z$.  As $(x,y)\to 0$, $B(x,y) = 1 + \mathcal{O}(x^2+y^2)$.
\item The metric $\dd\mr{s}^2$ is equal to ${\rm d}x^2+{\rm d}y^2+{\rm d}z^2+\mathfrak{h}$ where $\mathfrak{h}$ obeys:
\begin{enumerate}
\item $\mathfrak{h}(\pl_z,\pl_z) = 0$;
\item as a symmetric $2$-tensor measured by $\dd x^2+\dd y^2$, the restriction $\mathfrak{h}|_{z=0} = \mathcal{O}(x^2+y^2)$ as $(x,y)\to 0$.
\end{enumerate}
\item Since $\varphi$ is an embedding, the image of the disks $C_{\ell_a}\times\{z\}$ are transverse to the Reeb vector field $v$ for any $z\in I_{\ell_a}$.  These disks are called the \emph{transverse disks}, and are denoted by $C_z$.
\end{enumerate}

\subsubsection{Dirac operator in adapted coordinate chart}\label{sec_adapted_Dirac}
In this step, we introduce a \emph{transverse-Reeb exponential gauge} to trivialize the bundle $K^{-1}$ and $E$.  The exponential coordinate and exponential gauge is a standard trick in differential geometry and gauge theory.  The detail of the computation will be presented in \S\ref{apx_adapted_chart}.

Consider the adapted metric $\dd\mr{s}^2$.  Parallel transport $e_1$ and $e_2$ along radial geodesics on $C_0$.  Denote the resulting vector fields by $u_1$ and $u_2$.  They are linearly independent with the Reeb vector field $v$, but need not to be orthonormal.  The Gram--Schmidt process on $\{v,u_1,u_2\}$ produces an orthonormal frame $\{v,e_1,e_2\}$ on $C_0$.  Note that the Gram--Schmidt process does nothing at $T_pY$, and the notation is consistent.  Then parallel transport $\{v, e_1,e_2\}$ along the integral curves of $v$.  It ends up with a smooth, orthonormal frame on the adapted chart.  Denote the frame by by $\{v,e_1,e_2\}$.  The unit-normed section $\frac{1}{\sqrt{2}}(e_1 - ie_2)$ trivialize the bundle $K^{-1}$.  The bundle $E$ is trivialized in a similar way:  start with any unit-normed section at $p$, parallel transport along radial geodesic on $C_0$, and then parallel transport along the integral curves of $v$.  Since $E$ is a line bundle, the trivialization of $E$ does not require the Gram--Schmidt process.

With such a unitary trivialization of $E\oplus K^{-1}E$, the sections $\mr{\alpha}$ and $\mr{\beta}$  are identified with complex valued functions on $C\times I$.  Remember that $v=\pl_z$.  The expression of $e_1$ and $e_2$ in $\pl_x$, $\pl_y$ and $\pl_z$ can be found by the standard Jacobi field computation.  The Dirac operator takes the following form:
\begin{align}\label{eqn_Dirac_00}\left\{\begin{aligned}
\pr_1(\mr{D}_r\mr{\psi}) &= \frac{r}{2}\mr{\alpha} + i\pl_z\mr{\alpha} + \mu_0\mr{\alpha} + \dbrs\mr{\beta} - i\bar{\mu}_1\pl_z\mr{\beta} +\bar{\mu}_2\mr{\beta} ~,\\
\pr_2(\mr{D}_r\mr{\psi}) &= \dbr\mr{\alpha} - i\mu_1\pl_z\mr{\alpha} + \mu_2\mr{\alpha} - (\frac{r}{2}+c_0)\mr{\beta} - i\pl_z\mr{\beta} + \mu_3\mr{\beta}
\end{aligned}\right.\end{align}
where $\dbr$ and $\dbrs$ consist of taking derivatives in $x$ and $y$, but not in $z$.

Besides the $\pm{r}/{2}$ terms, all the other terms are independent of $r$.  Namely, they depend only on the contact form $a$, the adapted metric $\dd\mr{s}^2$ and the connection $A_0$.  The coefficients $\mu_0$ and $\mu_3$ are real-valued smooth functions, and $\mu_1$ and $\mu_2$ are complex-valued functions.

The operators $\dbr$ and $\dbrs$ are first order elliptic operators on $C_z$.  In other words, they are a smooth family of Cauchy--Riemann operators.  $\dbr$ and $\dbrs$ are \emph{almost adjoint} to each other in the following sense.  The volume form of the adapted metric $\dd\mr{s}^2$ is $\oh a\wedge\dd a = B{\rm d}x\wedge{\rm d}y\wedge{\rm d}z$.  Let $\omega = B{\rm d}x\wedge{\rm d}y$.  The self-adjointness of $\mr{D}_r$ and the $z$-independence of $B$ imply that
\begin{align}\label{eqn_selfadjoint_00}
\int_{C_z}\big( \langle\dbr\mr{\alpha},\mr{\beta}\rangle - \langle\mr{\alpha},\dbrs\mr{\beta}\rangle \big)\omega = \int_{C_z} -i(\pl_z\mu_1)\langle\mr{\alpha},\mr{\beta}\rangle\omega
\end{align}
for any $z\in I$ and any $\mr{\alpha}$ and $\mr{\beta}$ with compact support in $C_z$.

On the zero slice $C_0$, the frame $\{e_1,e_2\}$ differs from the usual exponential frame $\{u_1,u_2\}$ by the Gram--Schmidt process, which leads to a $\mathcal{O}(\sqrt{x^2+y^2})$ difference.  By the standard expansion in the exponential gauge, the coefficients of (\ref{eqn_Dirac_00}) on $C_0$ satisfies
\begin{enumerate}
\item $|\mu_j|\leq c_{45}\sqrt{x^2+y^2}$ for $j=0,1,2,3$;
\item $\dbr = \pl_x + i\pl_y + \mu_4\pl_x + \mu_5\pl_y$ where $\mu_4$ and $\mu_5$ are complex-valued functions which are also bounded by $c_{45}\sqrt{x^2+y^2}$.
\end{enumerate}
The constant $c_{45}$ is determined by the contact form $a$, the adapted metric $\dd\mr{s}^2$ and the connection $A_0$.

\subsubsection{Integral estimate over a transverse disk}

For any $p\in Y$, define $S(p,{\psi};\epsilon)$ to be the $2$-dimensional integral
\begin{align*}  S(p,{\psi};\epsilon) = \int_{C_{0,\epsilon}}|\mr{\alpha}|^2\omega = \int_{C_{0,\epsilon}}|\alpha|^2\Omega^2\omega  \end{align*}
where $C_{0,\epsilon}$ is the geodesic disk $\{\sqrt{x^2+y^2}\leq\epsilon\}$ on $C_0$.

\begin{prop}\label{prop_estimate_01}
There exists a constant $c_{46}$ determined by the contact form $a$, the conformally adapted metric $\dd s^2$ and the connection $A_0$ such that the following holds.  For any $r\geq c_{46}$ and $1\leq\lambda\leq \oh r^\oh$, suppose that $\psi\in\mathcal{V}(r,\lambda)$.  Then
\begin{align*}
S(p,\psi;\epsilon) &\leq c_{46}(\lambda + r^{-\oh}\epsilon^{-1}\lambda + r^\oh\epsilon^2\lambda^2)\int_Y|\psi|^2
\end{align*}
for any $p\in Y$ and any $\epsilon\leq\frac{1}{4}\ell_a$.
\end{prop}
\begin{proof}
Let $\tilde{\chi}_\epsilon$ be a cut-off function which depends on $\tilde{\rho} = \sqrt{x^2+y^2}$ with $\tilde{\chi}_\epsilon(\tilde{\rho}) = 1$ for $\tilde{\rho}\leq\epsilon$ and $\tilde{\chi}_\epsilon(\tilde{\rho}) = 0$ for $\tilde{\rho}\geq2\epsilon$.  Apply (\ref{eqn_Dirac_00}) to compute the rate of change of slice integrals:
\begin{align}\label{eqn_slice_change}\begin{split}
\frac{\dd}{\dd z}\big(\int_{C_z}\tilde{\chi}_\epsilon|\mr{\alpha}|^2\omega\big) &= 2\int_{C_z} \re\big( - i\tilde{\chi}_\epsilon\langle\mr{\alpha},\dbrs\mr{\beta}\rangle +  \tilde{\chi}_\epsilon\mu_1\langle\mr{\alpha},\pl_z\mr{\beta}\rangle \\
&\qquad\qquad - i\tilde{\chi}_\epsilon\mu_2\langle\mr{\alpha},\mr{\beta}\rangle + i\tilde{\chi}_\epsilon\langle\mr{\alpha},\mr{D}_r\mr{\psi}\rangle \big)\omega ~, \\
\frac{\dd}{\dd z}\big(\int_{C_z}\tilde{\chi}_\epsilon|\mr{\beta}|^2\omega\big) &= 2\int_{C_z} \re\big( - i\tilde{\chi}_\epsilon\langle\dbr\mr{\alpha},\mr{\beta}\rangle + \tilde{\chi}_\epsilon\mu_1\langle\pl_z\mr{\alpha},\mr{\beta}\rangle \\
&\qquad\qquad - i\tilde{\chi}_\epsilon\mu_2\langle\mr{\alpha},\mr{\beta}\rangle + i\tilde{\chi}_\epsilon\langle \mr{D}_r\mr{\psi},\mr{\beta}\rangle \big)\omega ~.
\end{split}\end{align}
Let $\tilde{S}(z)$ to be the following integral
\begin{align}\label{eqn_energy_01}   \tilde{S}(z) = \int_{C_z}\tilde{\chi}_\epsilon\big(|\mr{\alpha}|^2-|\mr{\beta}|^2-2\re(\mu_1\langle\mr{\alpha},\mr{\beta}\rangle)\big)\omega~.\end{align}
Since $\dbr$ and $\dbrs$ are almost adjoint (\ref{eqn_selfadjoint_00}) to each other, (\ref{eqn_slice_change}) leads to the following gradient estimate:
\begin{align}
\big| \frac{\dd}{\dd z}\tilde{S}(z) \big| &\leq c_{47}\int_{C_z} \big( |\dd\tilde{\chi}_\epsilon||\mr{\alpha}||\mr{\beta}| + \tilde{\chi}_\epsilon|\mr{\psi}||\mr{D}_r\mr{\psi}| \big)\omega ~.
\end{align}
Its integration says that
\begin{align}
\big| \tilde{S}(w) - \tilde{S}(0) \big| &= \big| \int_0^w (\frac{\dd}{\dd z}\tilde{S}(z))\dd z \big| \notag \\
&\leq \int_0^w\int_{C_z} \big( |\dd\tilde{\chi}_\epsilon||\mr{\alpha}||\mr{\beta}| + \tilde{\chi}_\epsilon|\mr{\psi}||\mr{D}_r\mr{\psi}| \big)B\,\dd x\,\dd y\,\dd z \notag \\
&\leq c_{47}\int_Y \big(|\dd\tilde{\chi}_\epsilon||\mr{\alpha}||\mr{\beta}| + \tilde{\chi}_\epsilon|\mr{\psi}||\mr{D}_r\mr{\psi}|\big) a\wedge\dd a \notag\\
&\leq c_{48}(1 + r^{-\oh}\epsilon^{-1})\lambda\int_Y|\psi|^2 \label{eqn_estimate_07}
\end{align}
for any $w\in[-\ell_a,\ell_a]$.  The last inequality follows from Lemma \ref{lem_CV_beta}, (\ref{eqn_estimate_00}) and (\ref{eqn_conformal_03}).

The quantity $\tilde{S}(0)$ can be written as
$$  \tilde{S}(0) = \frac{1}{2\ell_a}\big( -\int_{-\ell_a}^{\ell_a} (\tilde{S}(z)-\tilde{S}(0))\dd z + \int_{-\ell_a}^{\ell_a}\tilde{S}(z)\dd z\big) ~.  $$
The first integral is bounded by (\ref{eqn_estimate_07}), and the second integral is automatically bounded by $\int_Y|\psi|^2$.  Hence,
$$   |\tilde{S}(0)| \leq c_{49}(1 + r^{-\oh}\epsilon^{-1})\lambda\int_Y|\psi|^2 ~.   $$  Since $\mu_1$ is uniformly bounded on $C_0$, we apply the triangle inequality and the Cauchy--Schwarz inequality on (\ref{eqn_energy_01}) to conclude that
\begin{align}\label{eqn_estimate_20}
\int_{C_0}\tilde{\chi}_\epsilon|\mr{\alpha}|^2\,\omega &\leq c_{50}(1 + r^{-\oh}\epsilon^{-1})\lambda\int_Y|\psi|^2 + c_{50}\int_{\{\tilde{\rho}\leq2\epsilon\}\subset C_0}|\mr{\beta}|^2\,\omega ~.
\end{align}
The last term is less than $c_{51}\epsilon^2\sup_Y|\beta|^2$.   According to Proposition \ref{prop_estimate_00} and Lemma \ref{lem_estimate_00},
$$   \sup_Y|\beta|^2\leq c_{51}(r^\oh + r^{-\oh}\lambda^4)\int_Y|\psi|^2\leq 2c_{51}r^\oh\lambda^2\int_Y|\psi|^2 ~.   $$
Plugging it into (\ref{eqn_estimate_20}) finishes the proof of the proposition.
\end{proof}

\subsection{Pointwise estimate on $\alpha$}
The main purpose of this subsection is to prove the pointwise estimate on $\alpha$.

\begin{proof}[Proof of Theorem \ref{thm_point_estimate_00}]
Let $p_0$ be the point where $|\alpha|$ achieves its maximum.  It suffices to estimate $\mr{\alpha}(p_0) = (\Omega\alpha)(p_0)$.  Let $x,y,z$ be the adapted coordinate at $p_0$, and let $\tilde{\rho}$ be $\sqrt{x^2+y^2}$.  Let $\tilde{\chi}_\epsilon(\tilde{\rho})$ be the (slice-wise) cut-off function as introduced in the proof of Proposition \ref{prop_estimate_01}.  The precise value of $\epsilon$ will be chosen later.

Multiply the first equation of (\ref{eqn_Dirac_00}) by $\mu_1$, and add it to the second equation.
\begin{align*}
\dbr\mr{\alpha} &= i\mu_1\pl_z\mr{\alpha} - \mu_2\mr{\alpha} - \pr_2(\mr{D}_r\mr{\beta}) + \pr_2(\mr{D}_r\mr{\psi}) \\
&= - \mu_1\big( \frac{r}{2}\mr{\alpha} + \mu_0\mr{\alpha} + \pr_1(\mr{D}_r\mr{\beta}) - \pr_1(\mr{D}_r\mr{\psi}) \big) \\
&\qquad - \mu_2\mr{\alpha} - \pr_2(\mr{D}_r\mr{\beta}) + \pr_2(\mr{D}_r\mr{\psi}) ~.
\end{align*}
According to (\ref{eqn_Dirac_00}) and the discussion in \S\ref{sec_adapted_Dirac}, the restriction of the equation on the slice $C_0$ reads:
\begin{align}\label{eqn_Dirac_01} \begin{split}
(\pl_x + i\pl_y)(\tilde{\chi}_\epsilon\mr{\alpha}) &= \big((\pl_x + i\pl_y)(\tilde{\chi}_\epsilon)\big)\mr{\alpha} - \tilde{\chi}_\epsilon\big((\mu_4\pl_x + \mu_5\pl_y)\mr{\alpha}\big) \\
&\quad - \tilde{\chi}_\epsilon(\mu_1\frac{r}{2}+\mu_0\mu_1+\mu_2)\mr{\alpha} - \tilde{\chi}_\epsilon(\pr_2 + \mu_1\pr_1)(\mr{D}_r\mr{\beta}) \\
&\qquad + \tilde{\chi}_\epsilon\pr_2(\mr{D}_r\mr{\psi}) + \tilde{\chi}_\epsilon\mu_1\pr_1(\mr{D}_r\mr{\psi}) ~.
\end{split} \end{align}
The value of $\mr{\alpha}$ at $p_0$ can be found by the Cauchy integral formula for smooth functions.  It is equal to the integral of the right hand side of (\ref{eqn_Dirac_01}) against
\begin{align*}
-\frac{\dd x\wedge\dd y}{4\pi(x+iy)} \qquad\text{ over the disk }~ C_0 ~.
\end{align*}
The area element $\dd x\wedge\dd y=\frac{1}{B}\omega$ is uniformly equivalent to $\omega = B\dd x\wedge\dd y$.  Due to the uniformly equivalence, the crucial term is the factor $1/(x+iy)$.

We divide the right hand side of (\ref{eqn_Dirac_01}) into six terms.  Their Cauchy integrals are estimated as follows.
\begin{enumerate}

\item By Proposition \ref{prop_estimate_01}, the Cauchy integral of the first term is no greater than
\begin{align*}
\Big|\int_{C_0} \big(\frac{(\pl_x + i\pl_y)(\tilde{\chi}_\epsilon)}{x+iy}\mr{\alpha}\big)\Big| &\leq c_{55}\big(\int_{C_0}\frac{|\dd\tilde{\chi}_\epsilon|^2}{|\tilde{\rho}|^2}\big)^\oh \big(\int_{C_0}\tilde{\chi}_{2\epsilon}|\mr{\alpha}|^2\big)^\oh \\
&\leq c_{56}\,\epsilon^{-1}(\lambda^\oh + r^{-\frac{1}{4}}\epsilon^{-\oh}\lambda^\oh + r^{\frac{1}{4}}\epsilon\lambda) \big(\int_Y|\psi|^2\big)^\oh ~.
\end{align*}

\item After integration by parts, the Cauchy integral of the second term can be estimated by the same argument.  It is less than or equal to
\begin{align*}
&c_{57}\int_{C_0}\frac{\tilde{\chi}_\epsilon}{|\tilde{\rho}|}(1+|\dd\mu_4|+|\dd\mu_5|)|\mr{\alpha}| + c_{57}\big(\int_{C_0}\frac{(|\mu_4|^2+|\mu_5|^2)|\dd\tilde{\chi}_\epsilon|^2}{|\tilde{\rho}|^2}\big)^\oh \big(\int_{C_0}\tilde{\chi}_{2\epsilon}|\mr{\alpha}|^2\big)^\oh \\
\leq\,& c_{58}\,\epsilon\sup_Y|\alpha| + c_{58}(\lambda^\oh + r^{-\frac{1}{4}}\epsilon^{-\oh}\lambda^\oh + r^{\frac{1}{4}}\epsilon\lambda) \big(\int_Y|\psi|^2\big)^\oh ~.
\end{align*}

\item By the Cauchy--Schwarz inequality and Proposition \ref{prop_estimate_01}, the Cauchy integral of the third term is no greater than
\begin{align*}
&c_{59}\,r\big(\int_{C_0}\frac{\tilde{\chi}_\epsilon|\mu_1|^2}{|\tilde{\rho}|^2}\big)^\oh \big(\int_{C_0}\tilde{\chi}_\epsilon|\mr{\alpha}|^2\big)^\oh \\
\leq\,& c_{60}\,r\epsilon (\lambda^\oh + r^{-\frac{1}{4}}\epsilon^{-\oh}\lambda^\oh + r^{\frac{1}{4}}\epsilon\lambda) \big(\int_Y|\psi|^2\big)^\oh ~.
\end{align*}

\item To estimate the fourth term, note that $|D_r\beta|\leq|\nabla_r\beta|$.  Invoke Proposition \ref{prop_estimate_02} and Lemma \ref{lem_estimate_00} to bound $\sup|\nabla_r\beta|$.  The Cauchy integral of the fourth term is less than or equal to
\begin{align*}
c_{61}(\sup_Y|\nabla_r\beta|)\int_{C_0}\frac{|\tilde{\chi}_\epsilon|}{|\tilde{\rho}|} &\leq c_{62}\,\epsilon(r^{\frac{3}{4}} + r^{\frac{1}{4}}\lambda^2) \big(\int_Y|\psi|^2\big)^\oh ~.
\end{align*}

\item Since $D_r\psi$ still belongs to $\CV(r,\lambda)$, we can apply Proposition \ref{prop_estimate_00}, Corollary \ref{cor_estimate_00} and (\ref{eqn_estimate_00}) to bound $\sup|\pr_2(D_r\psi)|$.  The Cauchy integral of the fifth term is no greater than
\begin{align*}
c_{63}(\sup_Y|\pr_2(D_r\psi)|) \int_{C_0}\frac{|\tilde{\chi}_\epsilon|}{\tilde{\rho}} &\leq c_{64}\,\epsilon(r^{\frac{1}{4}}\lambda + r^{-\frac{1}{4}}\lambda^3) \big(\int_Y|\psi|^2\big)^\oh ~.
\end{align*}

\item With the help of Corollary \ref{cor_estimate_00}, the Cauchy integral of the last term is less than or equal to
\begin{align*}
c_{65}(\sup_Y|D_r\psi|)\int_{C_0}\frac{\tilde{\chi}_\epsilon|\mu_1|}{|\tilde{\rho}|} &\leq c_{66}\,\epsilon^2(r^{\frac{3}{4}}\lambda) \big(\int_Y|\psi|^2\big)^\oh ~.
\end{align*}
\end{enumerate}
Set $\epsilon$ to be $r^{-\oh}$.  A straightforward computation on the above six estimates shows that
\begin{align} \sup_Y|\alpha| \leq c_{67}\big( r\lambda \int_Y|\psi|^2 \big)^\oh ~. \end{align}
With Proposition \ref{prop_estimate_00}, it completes the proof of Theorem \ref{thm_point_estimate_00}.
\end{proof}

\section{The Heat Kernel}\label{sec_heat_kernel}

Denote by $\pi_L$ and $\pi_R$ the respective projection from $(0,\infty)\times Y\times Y$ to the left and right hand factor of $Y$.  The heat kernel for $D_r^2$ is a smooth section of $\Hom(\pi_R^*\BS,\pi_L^*\BS)$ over $(0,\infty)\times Y\times Y$ given by
\begin{align}\label{eqn_heat_00}
H_r(t;p,q) &= \sum_j e^{-\lambda_j^2 t}\psi_j(p)\psi_j^\dagger(q)
\end{align}
where $\{\psi_j\}$ constitutes a complete, orthonormal basis of eigensections for $D_r$, and $\lambda_j$ is the corresponding eigenvalue.  As a function of $t$ and $p$ with $q$ fixed, the heat kernel obeys the equation
\begin{align}\label{eqn_heat_01}
\frac{\pl}{\pl t}H_r &= -D_r^2 H_r ~.
\end{align}
Moreover, the $t\to 0$ limit of $H_r$ exists as a bundle valued measure:
\begin{align}\label{eqn_heat_02}
\lim_{t\to 0} H_r(t;p,\,\cdot\,) &= \BI\,\delta_p(\,\cdot\,)
\end{align}
where $\BI$ is the identity homomorphism in $\End(\BS)$ and $\delta_p$ is the Dirac measure at $p$.  In other words, $\zeta(p) = \lim_{t\to0}\int_Y H_r(t;p,q)\zeta(q)\dd q$ for any $\zeta\in\CC^\infty(Y;\BS)$.

For any $q\in Y$, choose unitary identifications $E|_q\cong\BC$ and $EK^{-1}|_q\cong\BC$.  Consider the following smooth section of $\pi^*\BS$ over $(0,\infty)\times Y$:
\begin{align}\label{eqn_heat_03}
h_{r,q}(t;p) &= \sum_je^{-\lambda_j^2t}\overline{\alpha_j(q)}\psi_j(p) ~.
\end{align}
Roughly speaking, it is the `first column' of $H_r$.  In particular, it obeys that heat equation (\ref{eqn_heat_01}), and
\begin{align}\label{eqn_heat_04}
\lim_{t\to0^+}\int_Y\langle\zeta(p),h_{r,q}(t;p)\rangle\dd p &= \pr_1\zeta(q)
\end{align}
for any $\zeta\in\CC^\infty(Y;\BS)$.  Here $\pr_1$ is the projection onto $E|_q\cong\BR$.

\subsection{Integral estimate of the heat kernel}
There are standard parametrix techniques to generate small time asymptotic expansion of the heat kernel, see \cite[chapter 2]{ref_BGV} or \cite[section 2]{ref_Taubes_sf}.  In order to estimate the remainder term in the asymptotic expansion, it requires some estimate on the heat kernel.  The following proposition provides a $L^2$-estimate on the heat kernel.  One can compare it with \cite[Proposition 2.1]{ref_Taubes_sf}.

\begin{prop}\label{prop_heat_00}
There exists a constant $c_1$ determined by the contact form $a$, the conformally adapted metric $\dd s^2$ and the connection $A_0$ such that:
\begin{align*}
\int_Y |h_{r,q}(t;p)|^2\,\dd p &\leq c_1(r + rt^{-\oh} + t^{-\frac{3}{2}}e^{-\frac{1}{10}r\,t})
\end{align*}  for any $r\geq c_1$, $q\in Y$ and $t>0$.
\end{prop}

\begin{proof}
We may assume that $|\lambda_j|$ is non-decreasing in $j$.  Weyl's asymptotic formula (see \cite[Corollary 2.43]{ref_BGV}) says that $|\lambda_j|^2 = \CO(j^{\frac{1}{3}})$ as $j\to\infty$.  It follows that the $L^2$-integral of $h_{r,q}(t;p)$ can be computed term by term:
$$ \int_Y|h_{r,q}(t,p)|^2\dd p = \sum_j e^{-2\lambda_j^2t}|\alpha_j(q)|^2 $$
for any $t>0$.  Divide the summation into two parts: $|\lambda_j|<10$ and $|\lambda_j|\geq10$.  According to Corollary \ref{cor_main_thm_00}(i), the first part is less than or equal to $c_2r$.

For the second part, note that
\begin{align*}
t\sum_{i=0}^\infty e^{-2(k+i)t}= \frac{t}{1-e^{-2t}}e^{-2kt} \geq \oh e^{-2kt}
\end{align*}
for any $k\geq0$ and $t>0$.  By the trick of summation by parts,
\begin{align*} 
\sum_{k=100}^\infty \big(2te^{-2kt}\sum_{|\lambda_j|^2<k+1}|\psi_j(q)|^2\big) &\geq \sum_{k=100}^\infty\Big(2t\big(e^{-2kt}+e^{-2(k+1)t}+\cdots\big)\sum_{k\leq|\lambda_j|^2<k+1}|\psi_j(q)|^2\Big) \\
&\geq \sum_{k=100}^\infty\big(e^{-2kt}\sum_{k\leq|\lambda_j|^2<k+1}|\psi_j(q)|^2\big) \\
&\geq \sum_{|\lambda_j|\geq10}e^{-2\lambda_j^2t}|\alpha_j(q)|^2 ~.
\end{align*}
Hence, it suffices to estimate $\sum_{k=100}^\infty te^{-2kt}(\sum_{|\lambda_j|^2<k+1}|\psi_j(q)|^2)$.  When $k\leq[\frac{1}{10}r]$, apply Corollary \ref{cor_main_thm_00}(i) on $\sum_{|\lambda_j|^2<k+1}|\psi_j(q)|^2$; when $k>[\frac{1}{10}r]$, apply Corollary \ref{cor_estimate_01} on $\sum_{|\lambda_j|^2<k+1}|\psi_j(q)|^2$.  It follows that
\begin{align*}
\sum_{10\leq|\lambda_j|}e^{-2\lambda_j^2t}|\alpha_j(q)|^2 &\leq c_2t\big(\sum_{k=100}^{[\frac{1}{10}r]}e^{-2kt}rk^\oh\big) + c_2t\big(\sum_{k=[\frac{1}{10}r]}^\infty e^{-2kt}k^{\frac{3}{2}}\big) \\
&\leq c_3t\big(r\int_{100}^\infty e^{-2kt}k^\oh\dd k + \int_{\frac{1}{10}r}^\infty e^{-2kt}k^{\frac{3}{2}}\dd k\big) ~. 
\end{align*}
Note that
\begin{align*}
&\int_0^\infty e^{-2kt}k^\oh\dd k = (32)^{-\oh}\pi^\oh t^{-\frac{3}{2}} ~,~\text{ and } \\
&\int_{\frac{1}{10}r}^\infty e^{-2kt}k^{\frac{3}{2}}\dd k \leq e^{-\frac{rt}{10}}\int_0^\infty e^{-kt}k^{\frac{3}{2}}\dd k
= \frac{3}{4}\pi^\oh t^{-\frac{5}{2}} e^{-\frac{1}{10}rt} ~.
\end{align*}
Combining these estimates gives
$$   \sum_j e^{-2\lambda_j^2t}|\alpha_j(q)|^2 \leq c_4(r + rt^{-\oh} + t^{-\frac{3}{2}} e^{-\frac{1}{10}rt}) ~,   $$
which finishes the proof of the proposition.
\end{proof}

\subsection{Asymptotic expansion of the heat kernel}

\subsubsection{Local expression of $D_r^2$}\label{sec_local_Dr}
Consider the adapted metric $\dd\mr{s}^2 = \Omega^{-2}\dd s^2$ and the adapted coordinate at $q\in Y$.  With respect to the transverse-Reeb exponential gauge (\ref{eqn_Dirac_00}), the $r$-dependent terms of $D_r$ appear in the diagonal.  To compute the heat kernel of $D_r^2$, it is convenient to work with a gauge in which the $r$-dependent terms appear in the off-diagonal.

What follows explains such a gauge and the local expression of the Dirac operator.  The detail of the computation will appear in \S\ref{apx_adapted_chart}.  Consider the gauge transform
\begin{align*}
(\mb{\alpha},\mb{\beta}) = \exp\big(-\frac{i}{2}r(z + S(x,y))\big) (\mr{\alpha},\mr{\beta})
\end{align*}
where $S(x,y)$ is some $r$-independent quadratic polynomial in $x$ and $y$.  Basically, $S(x,y)$ is constructed from the linear term of $\mu_1$ in (\ref{eqn_Dirac_00}).  The gauge transform is defined only on the adapted chart.  With respect to this gauge, the Dirac operator $\mr{D}_r$ takes the following form:
\begin{align*}\left\{\begin{aligned}
\pr_1(\mr{D}_r\mb{\psi}) &= i\pl_z\mb{\alpha} - (\pl_x - i\pl_y)\mb{\beta} + \frac{r}{2}(x-iy)\mb{\beta} + \big(\sum_{j=1}^3\fb_1^j\pl_j\mb{\beta} + r\fb_4\mb{\beta} + \fb_2\mb{\beta}\big) ~,\\
\pr_2(\mr{D}_r\mb{\psi}) &= (\pl_x + i\pl_y)\mb{\alpha} + \frac{r}{2}(x+iy)\mb{\alpha} -i\pl_z\mb{\beta} + \big(-\sum_{j=1}^3\bar{\fb}_1^j\pl_j\mb{\alpha} + r\bar{\fb}_4\mb{\alpha} + \fb_3\mb{\alpha} - \fb_0\mb{\beta}\big) ~.
\end{aligned}\right.\end{align*}
where $\fb_{(\cdot)}$ are smooth functions on the adapted chart.  They satisfy
\begin{align}\label{eqn_fd_growth}
|\fb_0| &\leq c_5 ~,  &\sum_{j=1}^3|\fb_1^j| + |\fb_2| + |\fb_3| &\leq c_5|\bx| ~,  &|\fb_4| &\leq c_5|\bx|^2
\end{align}
where $|\bx| = (x^2+y^2+z^2)^\oh$.  The Dirac operator $\mr{D}_r$ is self-adjoint with respect to $\oh a\wedge\dd a$, which is $B\dd x\wedge\dd y\wedge\dd z$ on this adapted chart.

The local expression of $D_r$ can be derived by
$$   D_r\psi = \Omega^{-2}\mr{D}_r(\Omega\psi) = \Omega^{-1}\mr{D}_r\psi + \Omega^{-2}\mr{\cl}(\dd\Omega)\psi ~.   $$
Rescale $\psi$ by $\tilde{\psi} = (\Omega^{3}B)^\oh\psi$, and consider the operator
\begin{align}\label{eqn_fD_01}
\fD_r\tilde{\psi} = (\Omega^3 B)^{\oh} D_r\big( (\Omega^3 B)^{-\oh}\tilde{\psi} \big) ~.
\end{align}
Using the above expression of $\mr{D}_r$, the local expression of $\fD_r$ on $\tilde{\psi} = (\tilde{\alpha},\tilde{\beta})$ is
\begin{align}\label{eqn_heat_05}
\fD_r\tilde{\psi} &= \Omega^{-1}(q)\left[\begin{array}{l}
i\pl_z\tilde{\alpha} - (\pl_x - i\pl_y)\tilde{\beta} + \frac{r}{2}(x-iy)\tilde{\beta} \\
(\pl_x + i\pl_y)\tilde{\alpha} + \frac{r}{2}(x+iy)\tilde{\alpha} -i\pl_z\tilde{\beta}
\end{array}\right] + (r\fg_0+\fe_0)\tilde{\psi} + \sum_{j=1}^3\ff_0^j\pl_j\tilde{\psi}
\end{align}
where $\fe_0$, $\ff_0^j$ and $\fg_0$ are smooth $(2\times2)$ matrix-valued functions on the adapted chart.  In other words, we treat $\tilde{\psi} = (\tilde{\alpha},\tilde{\beta})\in\BC^2$ as a column vector, and those $(2\times2)$ matrices are endomorphisms of $\BC^2$.  These functions, $\fe_0$, $\ff_0^j$ and $\fg_0$, are determined by the contact form $a$, the metric $\dd s^2$ and the connection $A_0$; in particular, none depend on $r$.  Moreover, there exists a constant $c_6$ such that
\begin{align*}
|\fe_0| &\leq c_6 ~,    &\sum_{j=1}^3|\ff_0^j| &\leq c_6|\bx| ~,    &|\fg_0| &\leq c_6|\bx|^2
\end{align*}
where $|\bx| = (x^2+y^2+z^2)^\oh$.

Note that $\fD_r$ is self-adjoint with respect to the Euclidean measure $\dd x\,\dd y\,\dd z$ and the standard Hermitian pairing on $(\tilde{\alpha},\tilde{\beta})$.  The factor $(\Omega^3B)^\oh$ is used to normalize the measure, and this factor is usually referred as the half-density (\cite[p.65]{ref_BGV}).

The first term on the right hand side of (\ref{eqn_heat_05}) will be referred as the \emph{principal part} of $\fD_r$.  Let $\fL_r$ be the square of the principal part of $\fD_r$.  It is equal to
\begin{align}\label{eqn_heat_06}\left\{\begin{aligned}
\pr_1(\fL_r\tilde{\psi}) &= \Omega_q^{-2}\big(-\pl_z^2\tilde{\alpha} + \big(-4\pl_\xi\pl_{\bar{\xi}}\tilde{\alpha} + r\,\bar{\xi}\pl_{\bar{\xi}}\tilde{\alpha} - r\,\xi\pl_\xi\tilde{\alpha} + \frac{r}{4}|\xi|^2\tilde{\alpha}\big) - r\tilde{\alpha}\big) ~,\\
\pr_2(\fL_r\tilde{\psi}) &= \Omega_q^{-2}\big(-\pl_z^2\tilde{\beta} + \big(-4\pl_\xi\pl_{\bar{\xi}}\tilde{\beta} + r\,\bar{\xi}\pl_{\bar{\xi}}\tilde{\beta} - r\,\xi\pl_\xi\tilde{\beta} + \frac{r}{4}|\xi|^2\tilde{\beta}\big) + r\tilde{\beta}\big)
\end{aligned}\right.\end{align}
where $\Omega_q=\Omega(q)$, and $\xi$ is the complex coordinate $x+iy$.  Let $\fR_r = -\fD_r^2 + \fL_r$ be the remainder part of $-\fD_r^2$.  By squaring (\ref{eqn_heat_05}), $\fR_r$ has the following expression:
\begin{align}\label{eqn_heat_07}
\fR_r &= (\fe_2 + r\ff_2 + r^2\fh_2) + \sum_{j=1}^3(\fe_3^j + r\fg_3^j)\pl_j + \sum_{j,k=1}^3\ff_3^{jk}\pl_j\pl_k
\end{align}
where $\fe$, $\ff$, $\fg$ and $\fh$'s are $(2\times2)$ matrix-valued functions on the adapted chart.  They do not depend on $r$, and have the following significance:
\begin{align}\label{eqn_heat_08}
|\fe|&\leq c_7 ~,    &|\ff|&\leq c_7|\bx| ~,    &|\fg|&\leq c_7|\bx|^2 ~,    &|\fh|\leq c_7|\bx|^3
\end{align}
for all subscripts and superscripts.  It is not hard to see that $\fL_r$ is self-adjoint with respect to $\dd x\,\dd y\,\dd z$, and thus $\fR_r = -\fD_r^2 + \fL_r$ is also self-adjoint.

As a second order elliptic operator for $\BC^2$ valued functions on $\BR^3$, the heat kernel of $\fL_r$ is given by the Mehler's formula \cite[\S4.2]{ref_BGV}.  Let
\begin{align}\label{eqn_heat_09}\begin{split}
\kappa_r(t;(\xi_1,z_1),(\xi_2,z_2))  &= (4\pi)^{-\frac{3}{2}}\Omega_q{t^{-\oh}}\exp\big(-\frac{\Omega^2_q(z_1-z_2)^2}{4t}\big) \\
&\qquad \frac{r}{\sinh(\Omega^{-2}_q rt)}\exp\big(-\frac{r}{4}\coth(\Omega^{-2}_q rt)|\xi_1-\xi_2|^2 - \frac{r}{4}(\bar{\xi}_1\xi_2 - \xi_1\bar{\xi}_2)\big) ~.
\end{split}\end{align}
The function $\kappa_r$ is the heat kernel of (\ref{eqn_heat_06}) without the last term, $-r\tilde{\alpha}$ or $+r\tilde{\beta}$.  It follows that the heat kernel of $\fL_r$ is
\begin{align}\label{eqn_heat_10}
K_r(t;(\xi_1,z_1),(\xi_2,z_2)) = \kappa_r(t;(\xi_1,z_1),(\xi_2,z_2))
\left[\begin{array}{cc} e^{\Omega^{-2}_q rt} & 0 \\ 0 & e^{-\Omega^{-2}_q rt} \end{array}\right] ~.
\end{align}

\subsubsection{Trace of the heat kernel}

The first component of $h_{r,q}(t;p)$ at $p = q$ is canonically identified with a scalar, which is $\sum_j e^{-\lambda_j^2t}|\alpha_j(q)|^2$.  The following theorem studies its asymptotic expansion.

\begin{thm}\label{thm_heat_00}
There exists a constant $c_9$ determined by the contact form $a$, the conformally adapted metric $\dd s^2$ and the connection $A_0$ such that:
\begin{align*}
\Big|\big(\sum_j e^{-\lambda_j^2t}|\alpha_j(q)|^2\big) - \frac{1}{4\pi^{\frac{3}{2}}}\Omega_q^{-2}rt^{-\oh}\Big| &\leq c_9(t^{-\oh} + r^{\frac{9}{2}}t^4 + t^{-\frac{3}{2}}e^{-\oh rt} + r^{\frac{7}{2}}e^{-\frac{1}{c_9t}})
\end{align*}
for any $r\geq c_9$, $t\leq 1$ and $q\in Y$.
\end{thm}

\begin{proof}
(\emph{Step 1: the heat equation})\;  Let $x,y,z$ be the adapted coordinate centered at $q$.  Suppress the subscript $q$ in $h_{r,q}$ for brevity.  Let $\chi_0$ and $\chi$ to be the standard cut-off functions which depends on $|\bx| = (x^2+y^2+z^2)^\oh$ such that
\begin{align*}
&\begin{cases} \chi_0(|\bx|) = 1 &\text{when }|\bx|\leq\frac{1}{128}\ell_a ~, \\
\chi_0(|\bx|) = 0 &\text{when }|\bx|\geq\frac{1}{64}\ell_a ~, \end{cases} &
&\begin{cases} \chi(|\bx|) = 1 &\text{when }|\bx|\leq\frac{1}{32}\ell_a ~, \\
\chi(|\bx|) = 0 &\text{when }|\bx|\geq\frac{1}{16}\ell_a ~. \end{cases}
\end{align*}
Consider
$$  \tilde{h}_r = \chi_0\,h_{r}(\Omega^3 B)^\oh ~.  $$
With respect to the transverse-Reeb exponential gauge twisted by $\exp(-\frac{i}{2}r(z+S(x,y)))$ as in \S\ref{sec_local_Dr}, regard $\tilde{h}_r$ as a $\BC^2$ valued functions on $(0,\infty)\times\BR^3$.  Since $h_{r}$ obeys the heat equation, $\chi_0 h_r$ satisfies
\begin{align*}
\frac{\pl}{\pl t}(\chi_0 h_{r}) &= -\chi_0 D_r^2h_{r} = -D_r^2(\chi_0 h_{r}) + (\dd^*\dd\chi_0)h_{r} - 2\nabla_{\nabla\chi_0}h_{r} ~.
\end{align*}
Multiply it by $(\Omega^3B)^\oh$, and use (\ref{eqn_fD_01}) to obtain the heat equation for $\tilde{h}_r$:
\begin{align}
\frac{\pl}{\pl t}\tilde{h}_r = -\fD_r^2\tilde{h}_r + (\dd^*\dd\chi_0)h_{r}(\Omega^3 B)^\oh - 2(\Omega^3 B)^\oh\nabla_{\nabla\chi_0}h_{r} ~, \notag\\
\Rightarrow~\frac{\pl}{\pl t}\tilde{h}_r + \fL_r\tilde{h}_r = \chi\fR_r\tilde{h}_r + (\dd^*\dd\chi_0)h_{r}(\Omega^3 B)^\oh - 2 (\Omega^3 B)^\oh\nabla_{\nabla\chi_0}h_{r} ~.  \label{eqn_heat_11}
\end{align}
With the dummy factor $\chi$, the operator $\chi\fR_r$ is globally defined on $\BR^3$.  When $t\to0$, the condition (\ref{eqn_heat_04}) implies that
\begin{align}\label{eqn_heat_12}
\lim_{t\to0}\tilde{h}_r = \left[\begin{array}{c}\Omega_q^{-\frac{3}{2}}\,\delta_0(\,\cdot\,) \\ 0 \end{array}\right] ~.
\end{align}
where $\delta_0$ is the Dirac measure at the origin of $\BR^3$.  The measure on $\BR^3$ is the standard one, $\dd x\,\dd y\,\dd z$.

\smallskip
(\emph{Step 2: parametrix})\;  For any smooth, $\BC^2$ valued function $\varphi(t;\bx)$ on $(0,\infty)\times\BR^3$, define $\CK*\psi$ to be the following function
\begin{align}\label{eqn_heat_13}
(\CK*\psi)(t;\bx) &= \int_0^t\int_{\BR^3} K_r(s;\bx,\bx_1)(\chi\fR_r(\varphi))(t-s;\bx_1)\dd\bx_1\dd s
\end{align}
where $\bx = (x,y,z)$ and $\dd\bx$ is the standard measure on $\BR^3$.  Set $\mbr{k}_r(t;\bx)$ to be the following $\BC^2$ valued function
$$  \mbr{k}_r(t;\bx) = \big(\Omega_q^{-\frac{3}{2}}\kappa_r(t;\bx,0)\exp({\Omega_q^{-2}rt}),0\big) ~,  $$ and set $k_r(t;\bx)$ to be
\begin{align}\label{eqn_heat_14}
k_r(t;\bx) = \mbr{k}_r(t;\bx) + \int_0^t\int_{\BR^3} K_r(s;\bx,\bx_1)\big((\dd^*\dd\chi_0)h_{r}A^\oh - 2 A^\oh\nabla_{\nabla\chi_0}h_{r}\big)(t-s;\bx_1)\dd\bx_1\dd s ~.
\end{align}
Note that $\mbr{k}_r(t;\bx)$ solves $\frac{\pl}{\pl t}+\fL_r = 0$, and satisfies the initial condition (\ref{eqn_heat_12}).

By virtue of (\ref{eqn_heat_11}) and (\ref{eqn_heat_12}), the $\BC^2$ valued function $\tilde{h}_r$ obeys:
\begin{align}\label{eqn_heat_15}\begin{split}
\tilde{h}_r &= k_r + \CK*k_r + \CK*(\CK*\tilde{h}_r) \\
&= (1 + \CK*)(\mbr{k}_r) + \CK*(\CK*\tilde{h}_r) + (1 + \CK*)(k_r - \mbr{k}_r) ~.
\end{split}\end{align}
It suffices to examine the right hand side at $\bx = 0$ to prove the theorem.

\smallskip
(\emph{Step 3: Properties of $\kappa_r$})\;  In this step, we explain four ingredients for estimating the convolution operator $\CK*$.  These ingredients follow from straightforward computations, and the detail can be safely left to the reader.

Here is the first property.  For any non-negative integer $m$, there exists a constant $c'_m$ which is independent of $\bx_1,\bx_2\in\BR^3$ and $r, t>0$ such that
\begin{itemize}
\item $\big|(\pl_{\bx_1}^m\kappa_r)(t;0,\bx_1)\big| \leq c'_m(t^{-\frac{m}{2}} + r^{\frac{m}{2}})|\kappa_r(t;0,\frac{\bx_1}{2})|$ where $\pl_{\bx_1}$ means the first order derivative in any component of $\bx_1$;
\item $|\bx_1|^m\big|\kappa_r(t;0,\bx_1)\big| \leq c'_m\,t^{\frac{m}{2}}\,|\kappa_r(t;0,\frac{\bx_1}{2})|$;
\item $\big|(\pl_{\bx_2}\kappa_r)(t;\bx_1,\bx_2)\big| \leq c'_1\big(t^{-\oh} + r^{\oh} + r|\bx_1|\big) \big|\kappa_r(t;\frac{\bx_1}{2},\frac{\bx_2}{2})\big| $;
\item suppose that $\ff$ is a function on $\BR^3$ with $|\ff(\bx_1) - \ff(\bx_2)|\leq c_{10}|\bx_1-\bx_2|$ and $\ff(0)=0$, then
\begin{align*}
&\big|\ff(\bx_2)(\pl_{\bx_2}^2\kappa_r)(t;\bx_1,\bx_2) + \ff(\bx_1)(\pl_{\bx_1}\pl_{\bx_2}\kappa_r)(t;\bx_1,\bx_2) \big| \\
\leq\, & c_{10}c'_2\big( (t^{-\oh} + r^\oh)(1 + r|\bx_1|^2 + r|\bx_2|^2) + (r^2|\bx_1|^3 + r^2|\bx_2|^3) \big) \big|\kappa_r(t;\frac{\bx_1}{2},\frac{\bx}{2})\big| ~.
\end{align*}
\end{itemize}
These inequalities are based on the facts that $|s|^m\exp(-s^2)\leq c'_m\exp(-\frac{s^2}{2})$ and
\begin{align*}\begin{cases}
\frac{1}{c_{11}}(rt)^{-1}\leq \coth(\Omega_q^{-2}rt) \leq c_{11}(rt)^{-1} &\text{when }rt\leq 1 ~, \\
\frac{1}{c_{11}} \leq \coth(\Omega_q^{-2}rt) \leq c_{11} &\text{when }rt\geq 1 ~.
\end{cases}\end{align*}

What follows is the second property:  for any non-negative integers $m$ and $n$, there exists a constant $c''_{m,n}>0$ which is independent of $\bx_2\in\BR^3$ and $r,t>0$ such that
\begin{align}\label{eqn_heat_16}\begin{split}
&\int_0^t\Big(\int_{\BR^3}s^{\frac{m-1}{2}}\big|\kappa_r(s;\bx_1,\bx_2)\big|\,(t-s)^{\frac{n-1}{2}}\big|\kappa_r(t-s;0,\bx_1)\big| \dd\bx_1\Big)\dd s \\
\leq\,& c''_{m,n}t^{\frac{m+n}{2}} \big|\kappa_r(t;0,\frac{\bx_2}{2})\big| ~.
\end{split}\end{align}

The third property is an integral estimate on $\kappa_r e^{\Omega_q^{-2}rt}$ over $\BR^3$.  There exists a constant $c_{12}$ which is independent $r,t>0$ such that
\begin{align}\label{eqn_heat_17}
\int_{\BR^3}\big|\kappa_r(t;0,\bx) e^{\Omega_q^{-2}rt}\big|^2\dd\bx  \leq  c_{13}rt^{-\oh}\frac{e^{2\Omega_q^{-2}rt}}{\sinh(2\Omega_q^{-2}rt)}  \leq  c_{12}(rt^{-\oh} + t^{-\frac{3}{2}}) ~.
\end{align}
One can compare this estimate with Proposition \ref{prop_heat_00}.

The last property is about the $L^2$-integral of $\kappa_r$ \emph{away from the origin}.  For any non-negative integer $m$, there exists a constant $c'''_m$ which is independent of $r,t>0$ such that
\begin{align}\label{eqn_heat_18}\left\{\begin{aligned}
&\int_{|\bx|\geq\frac{1}{256}\ell_a}\big|\kappa_r(t;0,\bx)e^{\Omega_q^{-2}rt}\big|^2 \dd\bx \leq c'''_0(1+r^2t^2)e^{-\frac{1}{c'''_0t}} ~, \\
&\int_0^t\int_{|\bx|\geq\frac{1}{256}\ell_a} |\bx|^{-m}\,\big|\kappa_r(t-s;\bx,0)e^{\Omega_q^{-2}r(t-s)}\big|\,\big|\kappa_r(s;0,\bx)e^{\Omega_q^{-2}rs}\big| \dd\bx\dd s \leq c'''_m(1+r^2t^2)e^{-\frac{1}{c'''_mt}} ~.
\end{aligned}\right.\end{align}
These two inequalities are based on the fact that $|\kappa_r(t;0,\bx)e^{\Omega_q^{-2}rt}|\leq c_{14}(1+rt)t^{-\frac{3}{2}}\exp(-\frac{|\bx|^2}{8t})$.

\smallskip
(\emph{Step 4: asymptotics of $(1 + \CK*)(\mbr{k}_r)$})\;
The value of $\pr_1(\mbr{k}_r)$ at $\bx = 0$ is
\begin{align*}
\Omega_q^{-\frac{3}{2}}\kappa_r(t;0,0)e^{\Omega_q^{-2}rt} &= (4\pi)^{-\frac{3}{2}}\Omega_q^{-\oh}t^{-\oh}\frac{re^{\Omega_q^{-2}rt}}{\sinh(\Omega_q^{-2}rt)} \\
&= \frac{1}{4\pi^{\frac{3}{2}}}\Omega_q^{-\oh}rt^{-\oh} + (4\pi)^{-\frac{3}{2}}\Omega_q^{-\oh}rt^{-\oh}\frac{e^{-\Omega_q^{-2}rt}}{\sinh(\Omega_q^{-2}rt)} ~
\end{align*}
and hence
\begin{align}\label{eqn_heat_19}
\big|\pr_1(\mbr{k}_r)(t;0) - \frac{1}{4\pi^{\frac{3}{2}}}\Omega_q^{-\oh}rt^{-\oh}\big| &\leq c_{17}(t^{-\frac{3}{2}}e^{-\oh rt} + rt^{-\oh}e^{-\oh rt}) ~.
\end{align}

The value of $\pr_1(\CK*\mbr{k}_r)$ at $\bx = 0$ is
$$   \int_0^t\int_{\BR^3} e^{\Omega_q^{-2}rt} \kappa_r(t-s;\bx,0)(\chi\fR_r^{(1,1)}(\kappa_r))(s;\bx,0)\,\dd\bx\,\dd s   $$ where $\fR_r^{(1,1)}$ is the $(1,1)$-component of $\fR_r$.  To elaborate, note that all the terms in (\ref{eqn_heat_07}) has ``odd degree" leading order term except the $\fe_2$-term.  For instance, consider the term $r\ff_2$.  There exist constants $\grave{c}_1,\grave{c}_2,\grave{c}_3$ and $\grave{c}$ such that $|\ff_2^{(1,1)} - \sum_{j=1}^3 \grave{c}_j x_j|\leq \grave{c}|\bx|^2$ on the adapted chart.  Since $\int_{\BR^3}\big(\kappa_r(t-s;\bx,0)(\sum_{j=1}^3 \grave{c}_j x_j)\kappa_r(s;\bx,0)\big)\dd\bx = 0$,
\begin{align*}
&r\Big| \int_0^t\int_{\BR^3} e^{\Omega_q^{-2}rt}\kappa_r(t-s;\bx,0)\chi\ff_2^{(1,1)}\kappa_r(s;\bx,0)\,\dd\bx\,\dd s \Big| \\
\leq\,&r\int_0^t\int_{\BR^3}  (1-\chi) |\sum_{j=1}^3\grave{c}_jx_j|\,|\kappa_r(t-s;\bx,0)e^{\Omega_q^{-2}r(t-s)}|\,|\kappa_r(s;\bx,0)e^{\Omega_q^{-2}rs}|\,\dd\bx\,\dd s \\
&~ + \grave{c}r\int_0^t\int_{\BR^3} e^{\Omega_q^{-2}rt}|\bx|^2\,|\kappa_r(t-s;\bx,0)|\,|\kappa_r(s;\bx,0)|\,\dd\bx\,\dd s \\
\leq\,& c_{18}\big(rt^\oh(1 + r^2t^2)e^{-\frac{1}{c_{18}t}} + rt^2|\kappa_r(t;0,0)e^{\Omega_q^{-2}rt}|\big) \leq c_{19}(rt^{\oh} + r^{3}t^{\frac{5}{2}})
\end{align*}
By this trick and the properties in step 3,
\begin{align}\label{eqn_heat_20}\begin{split}
\big|\pr_1(\CK*\mbr{k}_r)(t;0)\big| &\leq c_{19}(t^{-\oh} + r^\oh + rt^\oh + r^{\frac{3}{2}}t + r^2t^{\frac{3}{2}} + r^{\frac{5}{2}}t^2 + r^{3}t^{\frac{5}{2}}) \\
& \leq c_{20}(t^{-\oh} + r^{3}t^{\frac{5}{2}} ) ~.
\end{split}\end{align}
The last inequality is obtained by considering whether $rt\geq 1$ or $rt\leq 1$.

\smallskip
(\emph{Step 5: estimate $\CK*(\CK*\tilde{h}_r)$})\;
Since $\fR_r$ is a self-adjoint operator, performing integration by parts leads to the following equation:
\begin{align}\label{eqn_heat_30}
(\CK*(\CK*\tilde{h}_r))(t;0) = \int_0^t\int_{\BR^3} (Q(s;\bx_2))^T\,\tilde{h}_r(t-s;\bx_2) \,\dd\bx_2\dd s ~,
\end{align}
where
\begin{align*}
Q(s;\bx_2) = \int_0^{s}\int_{\BR^3} \Big(\overline{\fR}_{r,\bx_2}\big(\chi(\bx_2) K_r(s_1;\bx_1,\bx_2)\big)\Big) \Big(\overline{\fR}_{r,\bx_1}\big(\chi(\bx_1) K_r(s-s_1;0,\bx_1)\big)\Big) \dd\bx_1\dd s_1 ~.
\end{align*}
Here $T$ means the transpose of the matrix, and $\overline{\fR}_r$ is (\ref{eqn_heat_07}) with all the coefficient functions being complex conjugated.

Let $q_1(s;\bx_2)$ be the first column of $Q(s;\bx_2)$.  With the first two properties of step 3, there exists a constant $c_{21}$ which is independent of $\bx\in\BR^3$ and $r,s>0$ such that
\begin{align*}
|q_1(s;\bx_2)| &\leq c_{21}(s + 1 + r^4s^4)\big|\kappa_r(s;0,\frac{\bx_2}{4})e^{\Omega_q^{-2}rs}\big| ~.
\end{align*}
By (\ref{eqn_heat_17}),
\begin{align}\label{eqn_heat_31}
\int_{\BR^3}|q_1(s;\bx_2)|^2\dd\bx_2 &\leq c_{22}(s + 1 + r^4s^4)^2(rs^{-\oh} + s^{-\frac{3}{2}}) ~.
\end{align}

It follows from the Cauchy--Schwarz inequality on (\ref{eqn_heat_30}) that
\begin{align*}
\big|\pr_1(\CK*(\CK*\tilde{h}_r))(t;0)\big| &\leq \int_0^t|\!|q_1(s;\bx)|\!|_{L^2(\BR^3)}\,|\!|\tilde{h}_r(t-s;\bx)|\!|_{L^2(\BR^3)}\,\dd s ~.
\end{align*}
Then invoke Proposition \ref{prop_heat_00} and (\ref{eqn_heat_31}) to conclude that
\begin{align}
&\big|\pr_1(\CK*(\CK*\tilde{h}_r))(t;0)\big| \notag\\
\leq\,& c_{23}\int_0^t(s + 1 + r^4s^4)(r^\oh s^{-\frac{1}{4}} + s^{-\frac{3}{4}})(r^\oh+ r^\oh(t-s)^{-\frac{1}{4}} + (t-s)^{-\frac{3}{4}})\dd s \notag\\
\leq\,& c_{24}\big((t^{\oh} + t^2) + (t^{-\oh} + r^{\frac{9}{2}}t^4)\big) ~. \label{eqn_heat_32}
\end{align}

\smallskip
(\emph{Step 6: estimate $(1 + \CK*)(k_r - \mbr{k}_r)$})\;  After performing integration by parts on the last term of (\ref{eqn_heat_14}) and applying the Cauchy--Schwarz inequality, $|\pr_1(k_r - \mbr{k}_r)(t;0)|$ is less than
\begin{align*}
c_{25}\int_0^t  \Big(\int_{{\rm supp}(\dd\chi)} \big| e^{\Omega_q^{-2}rs}(r + \pl_{\bx})(\kappa_r)(s;0,\bx)\big|^2 \dd\bx\Big)^\oh\, |\!|\chi{h}_r(t-s;\bx)|\!|_{L^2(\BR^3)}\,\dd s ~.
\end{align*}
According to Proposition \ref{prop_heat_00} and the properties in step 3,
\begin{align}\label{eqn_heat_33}
|\pr_1(k_r - \mbr{k}_r)(t;0)| 
&\leq c_{25}r^2 e^{-\frac{1}{c_{25}t}} ~.  
\end{align}
With the similar integration by parts argument,
\begin{align}\label{eqn_heat_34}
|\pr_1(\CK*(k_r - \mbr{k}_r))(t;0)| &\leq c_{26}r^{\frac{7}{2}}e^{-\frac{1}{c_{26}t}} ~.
\end{align}

\smallskip
(\emph{Step 7}).
All the terms on the right hand side of (\ref{eqn_heat_15}) have been estimated.  It follows from (\ref{eqn_heat_19}), (\ref{eqn_heat_20}), (\ref{eqn_heat_32}), (\ref{eqn_heat_33}) and (\ref{eqn_heat_34}) that
\begin{align*}
\big|\pr_1(\tilde{h}_r)(t;0) - \frac{1}{4\pi^{\frac{3}{2}}}\Omega_q^{-\oh}rt^{-\oh}\big| &\leq c_{27}(t^{-\oh} + r^{\frac{9}{2}}t^4 + t^{-\frac{3}{2}}e^{-\oh rt} + r^{\frac{7}{2}}e^{-\frac{1}{c_{27}t}}) ~.
\end{align*}
Since $h_{r,q}(t;q) = \Omega_q^{-\frac{3}{2}}\tilde{h}_r(t;0)$, this completes the proof of Theorem \ref{thm_heat_00}.
\end{proof}

\section{The Spectral Flow}\label{sec_sf}
For any $\rb\geq2$, let $\CE_{\rb}$ be the following configuration of eigenvalues:
\begin{align}\label{defeqn_sf_01}
\CE_{\rb} &= \big\{ (r,\lambda)\in\BR^2 ~\big|~ 1<r<\rb,\, |\lambda|^2<\frac{1}{9}r\text{ and }\lambda\text{ is an eigenvalue of }D_r \big\} ~.
\end{align}
According to \cite[\S5.1]{ref_Taubes_SW_Weinstein}, the set $\CE_{\rb}$ consists of continuous, piecewise smooth curves which have the following properties.
\begin{itemize}
\item These curves are mutually disjoint in the sense of counting multiplicities.  In particular, suppose that $(r,\lambda)\in\CE_{\rb}$ and $\dim\ker(D_r - \lambda\,\mathbb{I}) = k$, then there are exactly $k$ curves passing through $(r,\lambda)$.
\item The boundary of these curves satisfies $\lambda^2 = \frac{1}{9}r$ or $r\in\{1,\rb\}$.
\item These curves is parametrized by $r$.
\end{itemize}
There is no canonical way to construct these curves, but any method will suffice.  With this understood, we write $\CE_{\rb} = \{(r,\lambda_j(r))~|~1\leq j \leq J_{\rb}\}$ where $J_{\rb}$ is the total number of curves, and each $\lambda_j$ is a continues, piecewise smooth function defined over a sub-interval of $(1,\rb)$.

Let $t(r)$ be a positive, monotone decreasing, smooth function of $r$.  A specific choice of $t(r)$ will be made at the end of \S\ref{subsec_psi_dispalcement}.  With such a function, define an orientation preserving diffeomorphism from $\BR$ to $(-(\frac{\pi}{t(r)})^\oh,(\frac{\pi}{t(r)})^\oh)$ as follows:
\begin{align}\label{def_sf_02}
\Phi_r(\lambda) &= \int_0^{\lambda} e^{-u^2\,t(r)}\,\dd u ~.
\end{align}
Its rescaling defines an orientation preserving diffeomorphism from $[-\frac{1}{3}r^\oh,\frac{1}{3}r^\oh]$ to $[-\oh,\oh]$:
\begin{align}\label{def_sf_03}
\Psi_r(\lambda) &= \oh\frac{\Phi_r(\lambda)}{\Phi_r(\frac{1}{3}r^\oh)} ~.
\end{align}

We define the \emph{$\Psi$-displacement} of $\CE_{\rb}$ to be the following:
\begin{align}\label{def_sf_04}
\int_1^{\rb} \frac{\dd\Psi_r(\CE_{\rb})}{\dd r}\,\dd r
&= \sum_{j=1}^{J_{\rb}}\int_{\Dom(\lambda_j)} \frac{\dd\Psi_r(\lambda_j(r))}{\dd r}\,\dd r
\end{align}
where $\Dom(\lambda_j)\subset(1,\rb)$ is the domain of $\lambda_j(r)$.

The $\Psi$-displacement of $\CE_{\rb}$ is closely related to the spectral flow function $\sfa_a(\rb)$.  The behavior of the $\Psi$-displacement will be studied in detail in \S\ref{subsec_psi_dispalcement}.  In \S\ref{subsec_estimate_sf}, we will use the $\Psi$-displacement to estimate the spectral flow function.  \S\ref{subsec_base_conn} is a digression to discuss the effect of using different connections on $\det(\BS)$.

\subsection{The $\Psi$-displacement}\label{subsec_psi_dispalcement}

At a differentiable point of $\lambda_j(r)$, the integrand of (\ref{def_sf_04}) is
\begin{align*}
\frac{\dd\Psi_r(\lambda_j)}{\dd r} &= \frac{\lambda'_je^{-\lambda_j^2t}}{2\Phi_r(\frac{1}{3}r^\oh)} - \frac{\Phi_r(\lambda_j)r^{-\oh}e^{-\frac{1}{9}rt}}{12\big(\Phi_r(\frac{1}{3}r^\oh)\big)^2} \\
&\quad + \frac{\Phi_r(\lambda_j)\big(\int_0^{\frac{1}{3}r^\oh}u^2e^{-u^2t}\dd u\big) - \Phi_r(\frac{1}{3}r^\oh)\big(\int_0^{\lambda_j}u^2e^{-u^2t}\dd u\big)}{2\big(\Phi_r(\frac{1}{3}r^\oh)\big)^2} t'
\end{align*}
where prime means taking derivative in $r$.  After integration by parts, the numerator of the last term is equal to
\begin{align*}
&{\Phi_r(\lambda_j)\big(\int_0^{\frac{1}{3}r^\oh}u^2e^{-u^2t}\dd u\big) - \Phi_r(\frac{1}{3}r^\oh)\big(\int_0^{\lambda_j}u^2e^{-u^2t}\dd u\big)} \\
=\,& \frac{1}{2t}\Phi_r(\lambda_j)\big(\Phi_r(\frac{1}{3}r^\oh) - \frac{1}{3}r^\oh e^{-\frac{1}{9}rt}\big) - \frac{1}{2t}\Phi_r(\frac{1}{3}r^\oh)\big(\Phi_r(\lambda_j) - \lambda_j e^{-\lambda_j^2t}\big) \\
=\,& \frac{1}{2t}\big(\lambda_j e^{-\lambda_j^2t}\Phi_r(\frac{1}{3}r^\oh) - \frac{1}{3}r^\oh e^{-\frac{1}{9}rt}\Phi_r(\lambda_j)\big) ~.
\end{align*}
With the help of this computation, let
\begin{align}\label{def_sf_05}\left\{\begin{aligned}
\breve{\Psi}({\rb}) &= \oh\sum_{j=1}^{J_\rb}\int_{\Dom(\lambda_j)}{\big(\Phi_r(\frac{1}{3}r^\oh)\big)^{-1}}\big(\lambda_j'e^{-\lambda_j^2t}\big)\,\dd r ~, \\
\dot{\Psi}({\rb}) &= \frac{1}{4}\sum_{j=1}^{J_\rb}\int_{\Dom(\lambda_j)}{\big(\Phi_r(\frac{1}{3}r^\oh)\big)^{-1}t't^{-1}}\big(\lambda_je^{-\lambda_j^2t}\big)\,\dd r ~, \\
\ddot{\Psi}({\rb}) &= -\frac{1}{12}\sum_{j=1}^{J_\rb}\int_{\Dom(\lambda_j)}\big(\Phi_r(\frac{1}{3}r^\oh)\big)^{-1}(r^{-\oh} + r^\oh t^{-1}t')e^{-\frac{1}{9}rt}\Phi_r(\lambda_j)\,\dd r ~.
\end{aligned}\right.\end{align}
Then the $\Psi$-displacement of $\CE_{\rb}$ is equal to $\check{\Psi}({\rb})+\dot{\Psi}({\rb})+\ddot{\Psi}({\rb})$.

\begin{rmk}
The above integrals can be rewritten as
\begin{align*}
\sum_{j=1}^{J_{\rb}}\int_{\Dom(\lambda_j)}F(\lambda_j(r))\dd r &= \int_1^{\rb}\sum_{|\lambda_j|<\frac{1}{3}r^\oh}F(\lambda_j) \dd r ~.
\end{align*}
\end{rmk}

\subsubsection{Asymptotics of $\breve{\Psi}(\rb)$}
The purpose of this subsection is to estimate $\breve{\Psi}(\rb)$.  Before doing that, we have to estimate $\big(\Phi_r(\frac{1}{3}r^\oh)\big)^{-1}$.
\begin{lem}\label{lem_sf_101}
For any $r\geq1$ and $0<t<1$ satisfying $rt\geq 50$,
\begin{align*}
\Big| \big(\Phi_r(\frac{1}{3}r^\oh)\big)^{-1} - \big(\frac{4}{\pi}\big)^\oh t^\oh \Big| &\leq 6\,r^{-\oh}e^{-\frac{1}{9}rt} ~,
\end{align*}
and $\frac{1}{10}t^\oh\leq\big(\Phi_r(\frac{1}{3}r^\oh)\big)^{-1}\leq 10\,t^\oh$.
\end{lem}
\begin{proof}
The quantity $\Phi_r(\frac{1}{3}r^\oh)$ is equal to $(\frac{\pi}{4})^\oh t^{-\oh}\big(1 - (\frac{4}{\pi})^\oh\int_{\frac{1}{3}(rt)^\oh}^\infty e^{-v^2}\dd v\big)$.  By integration by parts,
$$ \int_{\frac{1}{3}(rt)^\oh}^\infty e^{-v^2}\dd v = \frac{3}{2}(rt)^{-\oh}e^{-\frac{1}{9}rt} - \oh\int_{\frac{1}{3}(rt)^\oh}^\infty v^{-2}e^{-v^2}\dd v \leq \frac{3}{2}(rt)^{-\oh}e^{-\frac{1}{9}rt} ~, $$
and the first assertion follows.  The second assertion is a direct consequence of the first one.
\end{proof}

The following proposition uses the heat kernel expansion to estimate the function $\breve{\Psi}(\rb)$.
\begin{prop}\label{prop_sf_102}
There exists a constant $c_1$ determined by the contact form $a$, the conformally adapted metric $\dd s^2$ and the connection $A_0$ with the following property.  Suppose that $t(r)$ satisfies $50r^{-1}<t(r)<1$ when $r\geq c_1$.  Then
\begin{align*}
\big|\breve{\Psi}(\rb) - \breve{\Psi}(c_1) + \frac{\rb^2}{32\pi^2}{\int_Y a\wedge\dd a} \big| \leq c_1\int_{c_1}^{\rb}\big((rt)^{\frac{9}{2}} + r e^{-\frac{1}{20} rt}\big)\dd r
\end{align*}
for any $\rb\geq 2c_1$.  (The function $t(r)$ is abbreviated as $t$.)
\end{prop}
\begin{proof}
By (\ref{eqn_basic_est_01}), the slope of $\lambda_j(r)$ is given by
\begin{align*}
\lambda_j'(r) = \oh\int_Y \Omega^{-1}_q\big(|\alpha_j(q)|^2 - |\beta_j(q)|^2\big)
\end{align*}
where $\Omega_q = \Omega(q)$.  It follows that
\begin{align}\label{eqn_sf_102}
\sum_{|\lambda_j|<\frac{1}{3}r^\oh}(\lambda'_je^{-\lambda^2_jt}) &= \oh\int_Y\Omega_q^{-1}\sum_{|\lambda_j|<\frac{1}{3}r^\oh}e^{-\lambda_j^2t}\big(|\alpha_j(q)|^2 - |\beta_j(q)|^2\big)
\end{align}
where $\{\psi_j = (\alpha_j,\beta_j)\}$ is a set of $L^2$-orthonormal eigensections.

By Corollary \ref{cor_estimate_01} and with the same argument as that for Proposition \ref{prop_heat_00},
\begin{align}\label{eqn_sf_101}
\sum_{|\lambda_\psi|\geq\frac{1}{3}r^\oh}e^{-\lambda_\psi^2t} 
&\leq \sum_{k=[\frac{1}{9}r]}^\infty t e^{-kt}\big(\#\{\lambda_\psi~|~\lambda_\psi^2< k+1\}\big) \leq c_2t^{-\frac{3}{2}}e^{-\frac{1}{20}rt}
\end{align}
where the summation is indexed by an orthonormal set of eigensections of $D_r$ with eigenvalue $|\lambda_\psi|\geq\frac{1}{3}r^\oh$.  It follows from Theorem \ref{thm_heat_00} and (\ref{eqn_sf_101}) that
\begin{align}\label{eqn_sf_103}\begin{split}
&\Big| \oh\int_Y\Omega_q^{-1}\big(\sum_{|\lambda_j|<\frac{1}{3}r^\oh}e^{-\lambda_j^2t}|\alpha_j(q)|^2\big)\dd q - \frac{1}{8\pi^{\frac{3}{2}}}rt^{-\oh}\int_Y\Omega_q^{-3} \Big| \\
\leq\,& c_3(t^{-\oh} + r^{\frac{9}{2}}t^{4} + t^{-\frac{3}{2}}e^{-\frac{1}{20} rt} + r^{\frac{7}{2}}e^{-\frac{1}{c_3 t}}) ~.
\end{split}\end{align}
Note that the volume form of $\dd s^2$ is $\oh \Omega^3 a\wedge\dd a$.  According to Proposition \ref{prop_beta_estimate_00}(i),
\begin{align}\label{eqn_sf_104}
\int_{Y}\Omega_q^{-1}\big(\sum_{|\lambda_j|<\frac{1}{3}r^\oh}e^{-\lambda_j^2t}|\beta_j(q)|^2\big) &\leq c_4 r^{-1} \int_Y\Omega_q^{-1}\big(\sum_{|\lambda_j|<\frac{1}{3}r^\oh}e^{-\lambda_j^2t}|\alpha_j(q)|^2\big) ~.
\end{align}

It follows from (\ref{eqn_sf_102}), (\ref{eqn_sf_103}) and (\ref{eqn_sf_104}) that
\begin{align*}
\big| \sum_{|\lambda_j|\geq\frac{1}{3}r^\oh}(\lambda_j'e^{-\lambda_j^2t}) - \frac{rt^{-\oh}}{16\pi^{\frac{3}{2}}}\int_Y a\wedge\dd a \big| &\leq c_4(t^{-\oh} + r^{\frac{9}{2}}t^{4} + t^{-\frac{3}{2}}e^{-\frac{1}{20} rt} + r^{\frac{7}{2}}e^{-\frac{1}{c_3 t}}) ~.
\end{align*}
This inequality and Lemma \ref{lem_sf_101} find a constant $c_5$ such that
\begin{align*}
\big| \sum_{|\lambda_j|\geq\frac{1}{3}r^\oh}\big(\frac{\lambda_j'e^{-\lambda_j^2t}}{2\Phi_r(\frac{1}{3}r^\oh)}\big) - \frac{r}{16\pi^{2}}\int_Y a\wedge\dd a \big| &\leq c_5\big( (rt)^{\frac{9}{2}} + r e^{-\frac{1}{20} rt} \big) ~.
\end{align*}
for any $r\geq c_5$ and $t\in(50r^{-1},1)$.  The upper bound has been simplified using the condition $t\geq50r^{-1}$.  Integrating the inequality against $\dd r$ completes the proof of the proposition.
\end{proof}

\subsubsection{Estimate $\dot{\Psi}(\rb)$}
If we simply consider the magnitude of the integrand of $\dot{\Psi}(\rb)$, we can only conclude that $\dot{\Psi}(\rb)$ is about of order $\rb^{\frac{3}{2}}$.  To proceed, note that the sign of the integrand of $\dot{\Psi}(\rb)$ depends on the sign of $\lambda$.  It suggests that the cancellation argument may lead to a better estimate.  In the following lemma, the `leading order terms' can be integrated (step 2 below), and cancel with each other (step 4 below).  However, this trick relies on the fact that $\lambda' = \oh + \CO(r^{-1})$, and only works for an adapted metric.

\begin{lem}\label{lem_sf_201}
Suppose that $\dd s^2$ is an adapted metric, namely $\Omega\equiv1$.  There exist constants $c_7$ and $c_8$ determined by the contact form $a$, the adapted metric $\dd s^2$ and the connection $A_0$ such that the following holds.  Suppose that $t(r)$ satisfies $50r^{-1}<t(r)<1$ when $r\geq c_7$.  Then
\begin{align*}
\big|\dot{\Psi}(\rb) - \dot{\Psi}(c_7) \big| &\leq c_7\Big( 1 + \sup\{|t''|+|t'|^2 : {c_7<r<c_7+c_8}\} \\
&\quad + c_7\,\rb\,\sup\big\{r^{\frac{5}{2}}(t^{-\oh}|t''| + rt^{-\oh}|t'|^2 + r^\oh|t'|e^{-\frac{1}{9}rt}) + r^{\frac{3}{2}} t^{-\oh}|t'| : {c_7<r<\rb} \big\} \Big) ~.
\end{align*}
for any $\rb\geq 2c_7$.  (The function $t(r)$ is abbreviated as $t$.)
\end{lem}
\begin{proof}
(\emph{Step 1: rewrite $\dot{\Psi}(\rb)$})\;
Let $c_9$ be a constant greater than the constants of Proposition \ref{prop_beta_estimate_00} and Corollary \ref{cor_main_thm_00}.  Since the metric is adapted, (\ref{eqn_slope_01}) says that $|\lambda_j'(r)-\oh|\leq c_9r^{-1}$ provided $\lambda_j(r)$ is differentiable at $r\in(c_9,\rb)$.

Granted what was said, consider the curves in the interior of $\CE_{\rb}\backslash{\CE_{4c_9}}$ for any $\rb\geq8c_9$.  For each curve $\lambda_j(r)$, denote its domain by $({\sr}_j,\hat{\sr}_j)\subseteq(4c_9,\rb)$.  Since $|\lambda'_j(r)-\oh|\leq{c_9}{r^{-1}}$ on the smooth strata and $\CE_{\rb}$ is constrained by $\lambda^2 = \frac{1}{9}r$, there exists a constant $c_{10}>0$ such that $\hat{\sr}_j - {\sr}_j \leq c_{10}{\sr}_j^\oh$.

Denote $t(\sr_j)$ by $\st_j$ and $t'(\sr_j)$ by $\st_j'$.  Rewrite the integral of $4\dot{\Psi}$ along $\lambda_j(r)$ as follows:
\begin{align}
&\int_{{\sr}_j}^{\hat{\sr}_j} \big(\Phi_r(\frac{1}{3}r^\oh)\big)^{-1}t^{-1}t'\lambda_j e^{-\lambda_j^2t}\dd r \notag \\
\label{eqn_sf_201}\begin{split} =\,& \int_{{\sr}_j}^{\hat{\sr}_j} \big(\Phi_{\sr_j}(\frac{1}{3}\sr_j^\oh)\big)^{-1}\st_j^{-1}\st_j'\lambda_j e^{-\lambda_j^2\st_j}(2\lambda_j')\dd r + \int_{{\sr}_j}^{\hat{\sr}_j} \big(\Phi_{\sr_j}(\frac{1}{3}\sr_j^\oh)\big)^{-1}\st_j^{-1}\st_j'\lambda_j e^{-\lambda_j^2\st_j}(1-2\lambda_j')\dd r \\
&\quad + \int_{{\sr}_j}^{\hat{\sr}_j} \Big(\big(\Phi_r(\frac{1}{3}r^\oh)\big)^{-1}t^{-1}t'\lambda_je^{-\lambda_j^2t} - \big(\Phi_{\sr_j}(\frac{1}{3}\sr_j^\oh)\big)^{-1}\st_j^{-1}\st_j'\lambda_je^{-\lambda_j^2\st_j}\Big)\dd r ~.
\end{split}\end{align}

\smallskip
(\emph{Step 2: estimate the integrals})\;
The first integral on the right hand side of (\ref{eqn_sf_201}) can be evaluated, and is equal to
\begin{align*}
\big(\Phi_{\sr_j}(\frac{1}{3}\sr_j^\oh)\big)^{-1}\st_j^{-2}\st_j'(e^{-(\lambda_j({\sr}_j))^2\st_j} - e^{-(\lambda_j(\hat{\sr}_j))^2\st_j}) ~.
\end{align*}
With the help of Lemma \ref{lem_sf_101}, its magnitude is no greater than
\begin{align}
10\st_j^{-\oh}|\st_j'|\,|(\lambda_j(\hat{\sr}_j))^2 - (\lambda_j({\sr}_j))^2| ~.
\end{align}

Since $|\lambda'_j(r) - \oh|\leq c_9r^{-1}$ and $\hat{\sr}_j - {\sr}_j\leq c_{10}{\sr}_j^\oh$, the magnitude of the second integral on the right hand side of (\ref{eqn_sf_201}) is less than
\begin{align}
c_{11} \big(\Phi_{\sr_j}(\frac{1}{3}\sr_j^\oh)\big)^{-1}\st_j^{-1}|\st_j'|\sr_j^{-\oh}\sup\big\{|\lambda_j(r)|:{{\sr}_j<r<\hat{\sr}_j}\big\}  &  \leq c_{12} \st_j^{-\oh}|\st_j'| ~.
\end{align}
The inequality uses Lemma \ref{lem_sf_101}.

To estimate the third integral on the right hand side of (\ref{eqn_sf_201}), note that
\begin{align*}
&\Big|\frac{\dd}{\dd r}\big(\big(\Phi_r(\frac{1}{3}r^\oh)\big)^{-1}t^{-1}t'\big)\Big| \\
\leq\,&  c_{13}\Big( \big(\Phi(\frac{1}{3}r^\oh)\big)^{-1}\big(t^{-1}|t''| + t^{-2}|t'|^2 + rt^{-1}|t'|^2\big) + \big(\Phi(\frac{1}{3}r^\oh)\big)^{-2}t^{-1}|t'|e^{-\frac{1}{9}rt} \Big)~,
\end{align*}
and
$$ |e^{-\lambda_j^2 t} - e^{-\lambda_j^2\st_j}| < \lambda_j^2|t-\st_j| \leq c_{14} \sr_j^{\frac{3}{2}}\sup\{|t'|:{{\sr}_j<r<\hat{\sr}_j}\} ~. $$
Using these estimates and Lemma \ref{lem_sf_101}, the third integral of (\ref{eqn_sf_201}) is less than
\begin{align}
c_{15}\sr_j^{\frac{3}{2}}\sup\big\{t^{-\oh}|t''| + t^{-\frac{3}{2}}|t'|^2 + rt^{-\oh}|t'|^2 + |t'|e^{-\frac{1}{9}rt}: {{\sr}_j<r<\hat{\sr}_j} \big\} ~.
\end{align}
The term $t^{-\frac{3}{2}}|t'|^2$ can be absorbed by $rt^{-\oh}|t'|^2$ when $rt\geq 50$.

It follows that the magnitude of (\ref{eqn_sf_201}) is less than
\begin{align}\label{eqn_sf_202}\begin{split}
&10\st_j^{-\oh}|\st_j'|\,|(\lambda_j(\hat{\sr}_j))^2 - (\lambda_j({\sr}_j))^2| \\
&\quad + c_{16}\sup\big\{r^{\frac{3}{2}}(t^{-\oh}|t''| + rt^{-\oh}|t'|^2 + |t'|e^{-\frac{1}{9}rt}) + t^{-\oh}|t'| : {{\sr}_j<r<\hat{\sr}_j} \big\} ~.
\end{split}\end{align}

\smallskip
(\emph{Step 3: sum up the estimates})\;
The curves in the interior of $\CE_{\rb}\backslash\CE_{4c_1}$ can be divided into three parts:
\begin{align*}
J_1 &= \{j~|~\sr_j = 4c_9\}~, &J_2 &=\{j~|~4c_9<\sr_j<\hat{\sr}_j<\rb\}~, &\text{ and }& &J_3 &= \{j~|~\hat{\sr}_j=\rb\} ~.
\end{align*}
It is clear that the cardinality of $J_1$ is independent of $\rb$.  Thus, the summation of (\ref{eqn_sf_202}) over $J_1$ is less than
\begin{align}\label{eqn_sf_203}
c_{17}\big(1+\sup\{|t''|+|t'|^2:{4c_9<r<4c_9 + 2c_{10}c_9^\oh}\}\big) ~.
\end{align}

\smallskip
(\emph{Step 4: sum over $J_2$})\;
For any $j\in J_2$, the endpoints of $\lambda_j(r)$ satisfy $\lambda^2 = \frac{1}{9}r$, and thus
\begin{align*}
\st_j^{-\oh}|\st_j'|\,\big|(\lambda_j(\hat{\sr}_j))^2 - (\lambda_j({\sr}_j))^2\big| & \leq c_{10}\st_j^{-\oh}|\st_j'|\sr_j^\oh ~.
\end{align*}
It follows that (\ref{eqn_sf_202}) is less than
\begin{align*}
c_{18}\sup\big\{r^{\frac{3}{2}}(t^{-\oh}|t''| + rt^{-\oh}|t'|^2 + r^\oh|t'|e^{-\frac{1}{9}rt}) + r^\oh t^{-\oh}|t'| : {\sr_j<r<\hat{\sr}_j} \big\} ~.
\end{align*}

It follows from $\frac{1}{4}<\lambda'_j(r)<\frac{3}{4}$ that there exists a unique $\mr{\sr}_j\in(\sr_j,\hat{\sr}_j)$ such that $\lambda_j(\mr{\sr}_j) = 0$ for each $j\in J_2$ .  Moreover, each $j\in J_2$ contributes to the spectral flow count with $+1$ at $\mr{\sr}_j$.  With this understood, Corollary \ref{cor_main_thm_00}(ii) implies that the cardinality of $\{j\in J_2~|~k\leq\mr{\sr}_j< k+1\}$ is less than $c_9k$.  It follows that the summation of (\ref{eqn_sf_202}) over $\{j\in J_2~|~k\leq\mr{\sr}_j< k+1\}$ is less than
\begin{align*}
& (c_{9}k)c_{18} \sup\big\{r^{\frac{3}{2}}(t^{-\oh}|t''| + rt^{-\oh}|t'|^2 + r^\oh|t'|e^{-\frac{1}{9}rt}) + r^\oh t^{-\oh}|t'|: {|r-k|\leq 2c_{10}k^\oh}, r<\rb \big\} \\
\leq\, & c_{19} \sup\big\{r^{\frac{5}{2}}(t^{-\oh}|t''| + rt^{-\oh}|t'|^2 + r^\oh|t'|e^{-\frac{1}{9}rt}) + r^{\frac{3}{2}} t^{-\oh}|t'| : {4c_9<r<\rb} \big\} ~.
\end{align*}
The inequality is obtained by pushing $k$ into the supremum.  By chopping $[4c_1,\rb]$ into sub-intervals of length about $1$, the summation of (\ref{eqn_sf_202}) over $J_2$ is less than
\begin{align}\label{eqn_sf_204}
c_{20}\rb\sup\big\{r^{\frac{5}{2}}(t^{-\oh}|t''| + rt^{-\oh}|t'|^2 + r^\oh|t'|e^{-\frac{1}{9}rt}) + r^{\frac{3}{2}} t^{-\oh}|t'|:{4c_9<r<\rb} \big\} ~.
\end{align}

\smallskip
(\emph{Step 5: sum over $J_3$})\;
For any $j\in J_3$, let $\lambda_j(\rb) = \lim_{r\to\rb}\lambda_j(r)$.  It is clear that $|{\lambda}_j(\rb)|\leq\frac{1}{3}\rb^{\oh}$.  Due to the properties of $\lambda_j(r)$ explained at the beginning of \S\ref{sec_sf}, $\{{\lambda}_j(\rb)~|~j\in J_3\}$ are exactly all the eigenvalues of $D_{\rb}$ between $(-\frac{1}{3}\rb^\oh,\frac{1}{3}\rb^\oh]$.  With this understood, Corollary \ref{cor_main_thm_00}(i) implies that the cardinality of $J_3$ is less than $c_{9}\rb^{\frac{3}{2}}$.  It follows that the summation of (\ref{eqn_sf_202}) over $J_3$ is less than
\begin{align}\label{eqn_sf_205}
c_{21}\rb^{\frac{3}{2}}\sup\big\{r^{\frac{3}{2}}(t^{-\oh}|t''| + rt^{-\oh}|t'|^2 + r^\oh|t'|e^{-\frac{1}{9}rt}) + r^\oh t^{-\oh}|t'| : {\rb-c_{10}\sqrt{\rb}<r<\rb} \big\} ~.
\end{align}

\smallskip
(\emph{Step 6})\;  Combining (\ref{eqn_sf_203}), (\ref{eqn_sf_204}) and (\ref{eqn_sf_205}) completes the proof of the lemma.
\end{proof}

When the metric is conformally adapted, we simply leave $\dot{\Psi}(\rb)$ as
\begin{align}\label{eqn_sf_206}
\frac{1}{4}\int_1^{\rb}\big(\Phi(\frac{1}{3}r^{\oh})^{-1}t't^{-1}\sum_{|\lambda_j|<\frac{1}{3}r^\oh}(\lambda_je^{-\lambda_j^2t})\big) \dd r ~.
\end{align}
In the sequel of this paper \cite{ref_Ts3}, we will focus on certain types of contact form, and (\ref{eqn_sf_206}) will be studied by other methods.


\subsubsection{Estimate $\ddot{\Psi}(\rb)$}
The integrand of $\ddot{\Psi}(\rb)$ contains a factor of $e^{-\frac{1}{9}rt}$, which makes it much easier to handle.

\begin{lem}\label{lem_sf_301}
There exists a constant $c_{22}$ determined by the contact form $a$, the conformally adapted metric $\dd s^2$ and the connection $A_0$ with the following significance.  Suppose that $t(r)$ satisfies $50r^{-1}<t(r)<1$ when $r\geq c_{22}$.  Then
\begin{align*}
\Big|\ddot{\Psi}(\rb) - \ddot{\Psi}(c_{22})\Big| &\leq c_{22}\int_{c_{22}}^{\rb} \big|r + r^2 t^{-1}t'\big|e^{-\frac{1}{9}rt}\,\dd r
\end{align*}
for any $\rb\geq 2c_{22}$.  (The function $t(r)$ is abbreviated as $t$.)
\end{lem}
\begin{proof}
According to Corollary \ref{cor_main_thm_00}(i),
\begin{align*}
\sum_{|\lambda_j|<\frac{1}{3}r^\oh}(\Phi_r(\frac{1}{3}r^\oh))^{-1}\Phi_r(\lambda_j) &\leq c_9r^{\frac{3}{2}}
\end{align*}
for any $r\geq c_9$, and the lemma follows.
\end{proof}

\subsubsection{Estimate the $\Psi$-displacement}
We now choose the function $t(r)$, and specify the asymptotic behavior of the $\Psi$-displacement as $\rb\to\infty$.

\begin{prop}\label{prop_sf_401}
There exists a constant $c_{25}$ determined by the contact form $a$, the conformally adapted metric $\dd s^2$ and the connection $A_0$ with the following significance.  Let $t(r)$ be a positive, monotone decreasing, smooth function, which is equal to $20 r^{-1}\log r$ when $r\geq c_{25}$.  Then, the $\Psi$-displacement associated with $t(r)$ satisfies
\begin{align*}
\Big| \big(\int_1^{\rb}\frac{\dd\Psi_r(\CE_{\rb})}{\dd r}\,\dd r\big) - \frac{\rb^2}{32\pi^2} \int_Y a\wedge\dd a \Big| &\leq c_{25}\Big(\rb(\log\rb)^{\frac{9}{2}} + \int_{c_{25}}^{\rb}\big(r^{-\frac{3}{2}}\log r\sum_{|\lambda_j|<\frac{1}{3}r^\oh}(\lambda_je^{-\lambda_j^2t})\big) \dd r \Big)
\end{align*}
for any $\rb\geq 2c_{25}$.  Moreover, if the metric is adapted ($\Omega\equiv1$), then
\begin{align*}
\Big| \big(\int_1^{\rb}\frac{\dd\Psi_r(\CE_{\rb})}{\dd r}\,\dd r\big) - \frac{\rb^2}{32\pi^2} \int_Y a\wedge\dd a \Big| &\leq c_{25}\rb(\log\rb)^{\frac{9}{2}}
\end{align*}
for any $\rb\geq 2c_{25}$. 
\end{prop}

\begin{proof}
We first consider the case when the metric is adapted.  Let $c_{26}$ be a constant greater than the constant given by Proposition \ref{prop_sf_102}, Lemma \ref{lem_sf_201} and Lemma \ref{lem_sf_301}.  According to Proposition \ref{prop_sf_102},
\begin{align*}
\big| \breve{\Psi}(\rb)-\breve{\Psi}(c_{26}) - \frac{\rb^2}{32\pi^2}\int_Y a\wedge\dd a \big| &\leq c_{27}{\rb}(\log\rb)^{\frac{9}{2}}
\end{align*}
for any $\rb\geq 2c_{26}$.  By Lemma \ref{lem_sf_201} and Lemma \ref{lem_sf_301},
\begin{align*}
\big| \dot{\Psi}(\rb)-\dot{\Psi}(c_{26}) \big| &\leq c_{28}\rb(\log\rb)^{\frac{3}{2}} ~, \\
\big| \ddot{\Psi}(\rb)-\ddot{\Psi}(c_{26}) \big| &\leq c_{28}\rb
\end{align*}
for any $\rb\geq 2c_{26}$.  Since the $\Psi$-displacement at $c_{26}$ is independent of $\rb$, the second assertion of the proposition follows.

\smallskip

When the metric is only conformally adapted, Proposition \ref{prop_sf_102} and Lemma \ref{lem_sf_301} still holds.  Instead of Lemma \ref{lem_sf_201}, we apply Lemma \ref{lem_sf_101} and (\ref{eqn_sf_206}) to estimate $\dot{\Psi}(\rb)$.  This completes the proof of the proposition.
\end{proof}

\subsection{Estimate the spectral flow}\label{subsec_estimate_sf}
The main purpose of this subsection is to analyze the difference between the spectral flow function and the $\Psi$-displacement.

\begin{prop}\label{thm_sf_401}
There exists a constant $c_{33}$ determined by the contact form $a$, the conformally adapted metric $\dd s^2$ and the connection $A_0$ such that the following holds.  Let $t(r)$ be a positive, monotone decreasing, smooth function, which is equal to $20 r^{-1}\log r$ when $r\geq c_{33}$.  Then,
\begin{align*}
\Big| \sfa_a(\rb) - \big(\int_1^{\rb}\frac{\dd\Psi_r(\CE_{\rb})}{\dd r}\,\dd r\big) - \dot{\eta}(\rb) \Big| &\leq c_{33}\rb
\end{align*}
for any $\rb\geq 2c_{33}$.  The function $\dot{\eta}(\rb)$ is defined by
$$ (\frac{80}{\pi})^\oh \rb^{-\oh}(\log \rb)^\oh \Big(\sum_{\psi\in\CV_{\rb}^+}\int^{\frac{1}{3}\rb^\oh}_{\lambda_\psi} e^{-20(\rb^{-1}\log \rb)u^2}\,\dd u - \sum_{\psi\in\CV_{\rb}^-}\int_{-\frac{1}{3}\rb^\oh}^{\lambda_\psi} e^{-20(\rb^{-1}\log \rb)u^2}\,\dd u\Big)$$
where $\CV_{\rb}^+$ consists of orthonormal eigensetions of $D_{\rb}$ whose eigenvalue belongs to $(0,\frac{1}{3}\rb^\oh)$, $\CV_{\rb}^-$ consists of orthonormal eigensetions of $D_{\rb}$ whose eigenvalue belongs to $(-\frac{1}{3}\rb^\oh,0)$, and $\lambda_\psi$ is the corresponding eigenvalue.
\end{prop}

\begin{proof}
(\emph{Step 1: $\sfa_a(\rb)$ and the number of curves in $\CE_{\rb}$})\;
Let $c_{34}$ be a constant such that $\frac{1}{10}c_{34}$ is greater than the constants given by Proposition \ref{prop_beta_estimate_00} and Corollary \ref{cor_main_thm_00}.  For any $\rb\geq4c_{34}$, consider the curves $\{\lambda_j(r)\}$ in the interior of $\CE_{\rb}\backslash{\CE_{c_{34}}}$.  For each curve $\lambda_j(r)$, denote its domain by $({\sr}_j,\hat{\sr}_j)\subseteq(c_{34},\rb)$.  These curves can be divided into three parts:
\begin{align*}
J_1 &= \{j~|~\sr_j = c_{34}\}~, &J_2 &=\{j~|~c_{34}<\sr_j<\hat{\sr}_j<\rb\}~, &\text{ and }& &J_3 &= \{j~|~\hat{\sr}_j=\rb\} ~.
\end{align*}
Also, let $J_3^+ = \{j\in J_3 ~|~ \lim_{r\to\rb}\lambda_j(r)>0\}$ and $J_3^- = \{j\in J_3 ~|~ \lim_{r\to\rb}\lambda_j(r)\leq0\}$.  It is clear that $J_3 = J_3^+\amalg J_3^-$.

Proposition \ref{prop_beta_estimate_00}(ii) implies that $\frac{7}{20}\leq\lambda'\leq\frac{9}{20}$ on the smooth strata of $\CE_{\rb}\backslash{\CE_{c_{34}}}$.  In particular, there are only positive zero crossings for the spectral flow between $c_{34}$ and $\rb$.  Set
\begin{align*}  Z(c_{34},\rb) = \{(r,k)\in\BR\times\BN ~|~ c_{34}<r<\rb,~  \dim\ker D_r = k \}  \end{align*}
to be the set of zero crossings between $(c_{34},\rb)$.  It follows that
\begin{align*}
-c_{35} \leq \sfa_a(\rb) - \#\{Z(c_{34},\rb)\} \leq c_{34}\rb + c_{35} ~.
\end{align*}
The $c_{34}\rb$ in the upper bound comes from the dimension of $\ker D_{\rb}$, which is bounded by $c_{34}\rb$ by Corollary \ref{cor_main_thm_00}(i).

According to the properties of $\lambda_j(r)$ described at the beginning of \S\ref{sec_sf}, there is an \emph{injective} map
$$ \CJ:Z(c_{34},\rb) \to J_1\amalg J_2\amalg J_3^+ \qquad \text{ such that }\quad \lambda_{\CJ(r,k)}(r) = 0 $$
for any $(r,k)\in Z(c_{34},\rb)$.  The map $\CJ$ may not be unique, but any choice will suffice.  Roughly speaking, $\CJ(r,k)$ is the curve of eigenvalues contributed to the zero crossing $(r,k)$.  Moreover, the map $\CJ$ is almost surjective, possibly except $J_1$.  It follows that
\begin{align*}
\big| \#\{Z(c_{34},\rb)\} - \#\{J_1\amalg J_2\amalg J_3^+\} \big| &\leq c_{36} ~.
\end{align*}
By the triangle inequality,
\begin{align}\label{eqn_sf_401}
\big| \sfa_a(\rb) - \#\{J_1\amalg J_2\amalg J_3^+\} \big| &\leq c_{37}\rb ~.
\end{align}

\smallskip
(\emph{Step 2: count $J_2$ and $J_3$ via the $\Psi$-displacement})\;
For any $j\in J_2$, the endpoints\footnote{To be more precise, $\lambda_j(\sr_j) = \lim_{r\to\sr_j^+}\lambda_j(r)$ and $\lambda_j(\hat{\sr_j}) = \lim_{r\to\hat{\sr}_j^-}\lambda_j(r)$.} of $\lambda_j(r)$, $(\sr_j,\lambda_j(\sr_j))$ and $(\hat{\sr}_j,\lambda_j(\hat{\sr}_j))$, obey $\lambda^2 = \frac{1}{3}r$.  Due to Proposition \ref{prop_beta_estimate_00}(ii), $\lambda_j(\sr_j)<0$ and $\lambda_j(\hat{\sr}_j)>0$ for any $j\in J_2$.  It follows that $\Psi_{\sr_j}(\lambda_j(\sr_j)) = -\oh$ and $\Psi_{\hat{\sr}_j}(\lambda_j(\hat{\sr}_j)) = \oh$, and hence
\begin{align}\label{eqn_sf_403}
\sum_{j\in J_2} \int_{\sr_j}^{\hat{\sr}_j}\frac{\dd\Psi_r(\lambda_j(r))}{\dd r}\,\dd r &= \#\{J_2\} ~.
\end{align}

For any $j\in J_3^+$, $\Psi_{\sr_j}(\lambda_j(\sr_j)) = -\oh$ and
\begin{align}\label{eqn_sf_404}
\int_{\sr_j}^{\rb}\frac{\dd\Psi_r(\lambda_j(r))}{\dd r}\,\dd r &= \Psi_{\rb}(\lambda_j(\rb))) + \oh = 1 - \big(\Phi_\rb(\frac{1}{3}\rb^\oh)\big)^{-1}\int_{\lambda_j(\rb)}^{\frac{1}{3}\rb^\oh} e^{-20(\rb^{-1}\log\rb)u^2}\,\dd u ~.
\end{align}
Similarly, for any $j\in J_3^-$, $\Psi_{\sr_j}(\lambda_j(\sr_j)) = -\oh$, and
\begin{align}\label{eqn_sf_405}
\int_{\sr_j}^{\rb}\frac{\dd\Psi_r(\lambda_j(r))}{\dd r}\,\dd r &= \Psi_{\rb}(\lambda_j(\rb))) + \oh = \big(\Phi_\rb(\frac{1}{3}\rb^\oh)\big)^{-1}\int_{-\frac{1}{3}\rb^\oh}^{\lambda_j(\rb)} e^{-20(\rb^{-1}\log\rb)u^2}\,\dd u ~.
\end{align}
Since $\frac{7}{20}\leq\lambda'\leq\frac{9}{20}$, $j\in J_3^+\mapsto \lambda_j(\rb)$ is a bijection between $J_3^+$ and the spectrum of $D_{\rb}$ between $(0,\frac{1}{3}\rb^\oh]$.  And $j\in J_3^-\mapsto \lambda_j(\rb)$ is a bijection between $J_3^-$ and the spectrum of $D_{\rb}$ between $(-\frac{1}{3}\rb^\oh,0]$.  With this understood, summing up (\ref{eqn_sf_404}) over $J_3^+$ and (\ref{eqn_sf_405}) over $J_3^-$ gives:
\begin{align}\label{eqn_sf_406}
\Big| \#\{J_3\} - \sum_{j\in J_3^+} \int_{\sr_j}^{\rb}\frac{\dd\Psi_r(\lambda_j(r))}{\dd r}\,\dd r - \dot{\eta}(\rb)\Big| &\leq c_{38}\rb ~.
\end{align}
The inequality uses Lemma \ref{lem_sf_101}, Corollary \ref{cor_main_thm_00}(i) and the fact that
$$ \int_0^\infty e^{-20(\rb^{-1}\log\rb)u^2}\,\dd u \leq c_{39}\rb^{\oh} ~. $$

The proposition follows from the triangle inequality on (\ref{eqn_sf_401}), (\ref{eqn_sf_403}) and (\ref{eqn_sf_406}).
\end{proof}

\begin{thm}\label{cor_sf_402}
Suppose that $\dd s^2$ is an adapted metric, i.e.\ $\Omega\equiv1$.  There exists a constant $c_{41}$ determined by the contact form $a$, the adapted metric $\dd s^2$ and the connection $A_0$ such that
\begin{align*}
\Big| \sfa_a(\rb) - \frac{\rb^2}{32\pi^2}\int_Y a\wedge\dd a - \dot{\eta}(\rb) \Big| &\leq c_{41}\rb(\log\rb)^{\frac{9}{2}} ~.
\end{align*}
for any $\rb\geq c_{41}$.  The function $\dot{\eta}(\rb)$ is defined in Theorem \ref{thm_sf_401}.  As a consequence,
\begin{align*}
\Big| \sfa_a(\rb) - \frac{\rb^2}{32\pi^2}\int_Y a\wedge\dd a \Big| &\leq c_{41}\rb^{\frac{3}{2}}(\log\rb)^{-\oh} ~.
\end{align*}
\end{thm}

\begin{proof}
The first assertion is a direct consequence of Proposition \ref{thm_sf_401} and Proposition \ref{prop_sf_401}.
With the first assertion, it suffices to estimate $\dot{\eta}(\rb)$ to prove the second assertion.  By Corollary \ref{cor_eigen_distribution},
\begin{align*}
& \rb^{-\oh}(\log \rb)^\oh \sum_{\psi\in\CV_{\rb}^+}\int^{\frac{1}{3}\rb^\oh}_{\lambda_\psi} e^{-20(\rb^{-1}\log \rb)u^2}\,\dd u \\
=\,& \sum_{\psi\in\CV_{\rb}^+}\int_{\rb^{-\oh}(\log\rb)^\oh\lambda_\psi}^{\frac{1}{3}(\log\rb)^\oh}e^{-20 s^2}\dd s \leq  c_{42}\rb \sum_{k=0}^{[\frac{1}{3}\rb^\oh]} \big(\int_{\rb^{-\oh}(\log\rb)^\oh k}^{\frac{1}{3}(\log\rb)^\oh}e^{-20 s^2}\dd s\big) \\
\leq\; & c_{42}\rb\,\big(\frac{1}{4}\sqrt{\frac{\pi}{5}} + \int_0^{\frac{1}{3}\rb^\oh}\int_{\rb^{-\oh}(\log\rb)^\oh k}^{\frac{1}{3}(\log\rb)^\oh}e^{-20 s^2}\dd s\,\dd k\big) \\
=\,& c_{42}\rb\,\big(\frac{1}{4}\sqrt{\frac{\pi}{5}} + \int_0^{\frac{1}{3}(\log\rb)^\oh}\int_{0}^{\rb^\oh(\log\rb)^{-\oh}s}e^{-20 s^2}\dd k\,\dd s\big)
\leq c_{43}\rb^{\frac{3}{2}}(\log\rb)^{-\oh} ~.
\end{align*}
Clearly, the same estimates holds for the summation over $\CV_{\rb}^-$.  This completes the proof of the theorem.
\end{proof}

This theorem says that the subleading order term of the spectral flow function is strictly less than $\CO(r^{\frac{3}{2}})$.  It improves Proposition 5.5 of \cite{ref_Taubes_SW_Weinstein} when $a$ is a contact form with an adapted metric $\dd s^2$.  Although the improvement is far from satisfactory, it confirms that the subleading order term is of ${\scriptscriptstyle\CO}(\rb^{\frac{3}{2}})$.  This suggests that $\dot{\eta}(\rb)$ should be smaller due to cancellation.  In the sequel of this paper \cite{ref_Ts3}, $\dot{\eta}(\rb)$ will be shown to be about $\CO(\rb)$ for certain types of contact forms in each isotopy class of contact structures.

\subsection{The base connections}\label{subsec_base_conn}
It requires a unitary connection $A_0$ on $\det(\BS)$ to define a Dirac operator on the spinor bundle $\BS$.  The main purpose of this subsection is to compare the spectral flow functions using different connections on $\det(\BS)$.

\begin{prop}
Suppose that $A_0$ and $A_1$ are two connections on $\det(\BS)$.  Then, there exists a constant $c_{45}$ determined by the contact form $a$, the conformally adapted metric $\dd s^2$ and the connections $A_0$ and $A_1$ such that
\begin{align*}
\big| \sfa_a(A_0,r) - \sfa_a(A_1,r) \big| &\leq c_{45} r
\end{align*}
for any $r\geq c_{45}$.
\end{prop}

\begin{proof}
Since the spectral flow only depends on the endpoints of the connection, the difference $\sfa_a(A_1,r) - \sfa_a(A_0,r)$ is equal to
\begin{align*}
(\text{spectral flow from }A_1\text{ to }A_0) + (\text{spectral flow from }A_0-ira\text{ to }A_1-ira) ~.
\end{align*}
The spectral flow from $A_1$ to $A_0$ is clearly independent of $r$.  Therefore, it suffices to show that the spectral flow from $A_0-ira$ to $A_1-ira$ is of $\CO(r)$.

Let $\tilde{D}_t$ be the Dirac operator associated to $(1-t)A_0+tA_1-ira$ for $t\in[0,1]$.  Suppose that $\lambda(t)$ is an eigenvalue of $\tilde{D}_t$ for $t\in[0,1]$, and is continuous, piecewise smooth in $t$.  By \cite[(5.4)]{ref_Taubes_SW_Weinstein},
\begin{align}
\lambda'(t) = \int_Y\langle\psi_t,\oh\cl(A_1-A_0)\psi_t\rangle
\end{align}
provided $\lambda(t)$ is differentiable at $t$, where $\psi_t$ is a unit-normed eigensection of $\tilde{D}_t$ with eigenvalue $\lambda(t)$.  It follows that
\begin{align}
|\lambda'(t)|\leq c_{46} = 1+\oh\sup_Y|A_1-A_0| ~.
\end{align}

We apply Corollary \ref{cor_main_thm_00} to $\tilde{D}_t$ for any $t\in[0,1]$.  The constant of Theorem \ref{thm_point_estimate_00} depends on the curvature of $(1-t)A_0+tA_1$ and the covariant derivative of the curvature, and does not blow up for $t\in[0,1]$.  As a result, there exists a constant $c_{47}$ determined by $a$, $\dd s^2$, $A_0$ and $A_1$ such that the total number of eigenvalues (counting multiplicity) of $\tilde{D}_t$ within $[-1,1]$ is less than $c_{47}r$ for any $r\geq c_{47}$ and any $t\in[0,1]$.  It follows that the spectral flow from $\tilde{D}_{t_0}$ to $\tilde{D}_{t_0 + ({1}/{(2c_{46})})}$ is less than $c_{47}r$.  Hence, the spectral flow from $A_0-ira$ to $A_1-ira$ is less than $3c_{46}c_{47}r$.  It completes the proof of this proposition.
\end{proof}

\appendix
\section{~}
\subsection{The Weitzenb\"ock formula for $\nabla_r\psi$}\label{apx_commuting}
The purpose of this subsection is to derive the following formula:  suppose that $V$ is a Hermitian vector bundle with a unitary connection $\BA$, then
\begin{align}\label{eqn_Weitzenbock_01}
\nabla_\BA^*\nabla_\BA\nabla_\BA\psi - \nabla_\BA\nabla_\BA^*\nabla_\BA\psi &= (\dd_\BA^*\BF_\BA)\psi - \nabla_\BA\psi\lrcorner (2\BF_\BA + {\rm Ricci}) ~.
\end{align}
for any section $\psi$ of $V$.  When $V$ is a spin-c bundle and $\BA$ is a fixed connection perturbed by $-\frac{i}{2}ra$, (\ref{eqn_Weitzenbock_01}) leads to (\ref{eqn_estimate_02}).

For simplicity, assume the Riemannian metric on the underlying manifold is flat.  Suppose that the connection is $\BA = \sum_j\BA_j\dd x^j $, then the curvature is
\begin{align*}  \BF_\BA = \oh\sum_{i,j}\BF_{ij}\dd x^i\wedge\dd x^j \quad \text{ where }~\BF_{ij} = \pl_i\BA_j - \pl_j\BA_i + [\BA_i,\BA_j] ~,  \end{align*}
and $\dd_\BA^*\BF_\BA = \sum_{i,j}(\pl_j\BF_{ij} + [\BA_j,\BF_{ij}])\dd x^i$.
Note that
\begin{align*}
\psi_{;i} &= \pl_i\psi + \BA_i\psi \quad\text{where semicolon means covariant derivative } \nabla_\BA ~,\\
\psi_{;ji} - \psi_{;ij} &= \BF_{ij}\psi ~, \\
\psi_{;jik} - \psi_{;ijk} &= (\pl_k\BF_{ij} + [\BA_k,\BF_{ij}])\psi + \BF_{ij}\psi_{;k} ~, \\
\psi_{;jik} - \psi_{;ikj} &= (\pl_k\BF_{ij} + [\BA_k,\BF_{ij}])\psi + \BF_{ij}\psi_{;k} + \BF_{kj}\psi_{;i} ~.
\end{align*}
It follows that the $\dd x^j$-component of $\nabla_\BA^*\nabla_\BA\nabla_\BA\psi - \nabla_\BA\nabla_\BA^*\nabla_\BA\psi$ is
\begin{align*}
-\sum_{i}\psi_{;jii} + \sum_i\psi_{;iij} &= -\sum_i(\pl_i\BF_{ij} + [\BA_i,\BF_{ij}])\psi - 2\sum_i\BF_{ij}\psi_{;i} ~.
\end{align*}
This proves (\ref{eqn_Weitzenbock_01}) for flat metric.

\subsection{Adapted coordinate and transverse-Reeb exponential gauge}\label{apx_adapted_chart}
The purpose of this subsection is to derive the local expression of the Dirac equation on the adapted coordinate chart.  Suppose that $a$ is a contact form on $Y$, and $\dd\mr{s}^2$ is an adapted metric.  Denote the Reeb vector field by $v$, and the Levi-Civita connection of $\dd\mr{s}^2$ by $\nabla$.

Fix a point $p\in Y$.  The construction of the adapted chart starts with two oriented, orthonormal vectors $e_1$ and $e_2$ for $\ker(a)|_p$.  The choice of $e_1$ and $e_2$ is not unique; there is a freedom of ${\rm SO}(2)\cong S^1$.  We will choose $e_1$ and $e_2$ to be the eigenvectors of a symmetric map defined from $\nabla v$.  This choice makes it easier to do the local computation.

\subsubsection{The choice of the frame}
Consider the map $\CN$ on $\ker(a)|_p$ defined by
\begin{align*}
\langle \CN(u_1), u_2 \rangle &= \langle \nabla_{u_1}v, J(u_2)  \rangle
\end{align*}
for any $u_1,u_2\in\ker(a)|_p$.  The pairing is the $\dd\mr{s}^2$ inner product, and $J$ is the rotation operator on $\ker(a)$ defined by $\dd a$ and $\dd\mr{s}^2$.

Let $\fe_1$ be a unit-normed vector on $\ker(a)|_p$, and let $\fe_2 = J(\fe_1)$.  It follows from $\dd*a=0$ that
\begin{align*}
\langle\nabla_{\fe_1}v,\fe_1\rangle + \langle\nabla_{\fe_2}v,\fe_2\rangle = 0 ~.
\end{align*}
It implies that $\CN$ is a symmetric operator.  Choose $e_1$ to be one of the unit-normed eigenvector of $\CN$, and denote its eigenvalue by $1+N$.  Namely,
\begin{align}\label{eqn_apx_N1}
N = \langle\CN(e_1) - e_1 , e_1 \rangle ~.
\end{align}
Another vector $e_2$ is taken to be $J(e_1)$.  By contracting $(e_1,e_2)$ with $\dd a = 2* a$, we find that
\begin{align*}
\langle\nabla_{e_1}v,e_2\rangle - \langle\nabla_{e_2}v,e_1\rangle = 2 ~.
\end{align*}
Equivalently, the trace of $\CN$ is $2$.  Thus,
\begin{align}\label{eqn_apx_N2}
-N = \langle\CN(e_2) - e_2 , e_2 \rangle ~.
\end{align}

\subsubsection{The adapted coordinate}
With $e_1$ and $e_2$ chosen, consider the adapted coordinate centered at $p\in Y$:
\begin{align*}  \begin{array}{cccl}
&C\times I &\to & Y \\
\varphi_0: &((x,y),0) &\mapsto & \exp_{p}(x e_1+y e_2) ~, \\
\varphi: &((x,y),z) &\mapsto & \exp_{\varphi_0(x,y)}(zv) ~.
\end{array}  \end{align*}
It follows from the construction that $\varphi(x,y,\,\cdot\,)$ is a integral curve of the Reeb vector field for any $x$ and $y$.  Therefore, the Reeb vector field $v = \pl_z$.  By (\ref{eqn_apx_N1}) and (\ref{eqn_apx_N2}), its covariant derivative at $p$ is
\begin{align}\label{eqn_apx_01}
(\nabla_{e_1}v)|_p &= (1+N) e_2 ~, &(\nabla_{e_2}v)|_p &= (-1+N) e_1 ~.
\end{align}
It follows from $\dd a = 2*a$ that $\nabla_v v$ vanishes identically.

Since $a(v) = 1$ and $\dd a(v, \,\cdot\,) = 0$, the contact form and its exterior derivative must be
\begin{align}\label{eqn_apx_00}\begin{split}
a &= \dd z + 2a_1(x,y)\dd x + 2a_2(x,y)\dd y ~,\\
\dd a &= 2(\pl_x a_2(x,y) - \pl_y a_1(x,y))\dd x\wedge\dd y ~.
\end{split}\end{align}
And the volume form is $\oh a\wedge\dd a = B(x,y)\,\dd x\wedge\dd y\wedge\dd z$, where $B(x,y) = \pl_x a_2 - \pl_y a_1$.

To proceed, consider the following frame:  parallel transport $\{e_1,e_2,v\}$ along radial geodesics on $C_0$, and then parallel transport along the Reeb chords.  It ends up with an orthonormal frame on $C\times I$, which will be denoted by $\{u_1,u_2,u_3\}$.  We are going to find the transition between $\{u_1,u_2,u_3\}$ and $\{\pl_z,\pl_y,\pl_z\}$.

\subsubsection{The Reeb vector field}
To express $\pl_z$ in terms of $\{u_1,u_2,u_3\}$, note that both $\pl_z = e_3$ and $u_j$ are parallel along the integral curves of $v$.  Therefore, $\langle e_3,u_j\rangle$ is independent of $z$, and it suffices to compute these coefficients on $C_0$.  For any $(x,y)\in C$, consider the radial geodesic $\varphi_0(tx,ty)$.  Let $e_3|_{(tx,ty,0)} = \sum_j h_3^j(t)u_j$, then $\frac{\dd^k}{\dd t^k}h_3^j(t) = \langle(\nabla^{k}e_3)(\pl_t,\cdots,\pl_t),u_j\rangle$.  The Taylor's theorem and (\ref{eqn_apx_01}) imply that
\begin{align}\label{eqn_apx_02}
\pl_z &= u_3 + y(-1+N)u_1 + x(1+N)u_2 + \CO(\rho_0^2)u_j
\end{align}
where $\rho_0 = (x^2+y^2)^\oh$.

\subsubsection{The vector fields $\pl_x$ and $\pl_y$ on the zero slice}
Fix $(x,y)\in C$, and let $\gamma(t,s) = \varphi_0(t(x+s),ty)$.  Denote the variational field $\frac{\pl}{\pl s}|_{s=0}\gamma(t,s)$ by $V(t)$.  It follows from the construction that $V(1) = \pl_x|_{(x,y,0)}$.  Since $V(t)$ is a variational field of geodesics, it obeys the Jacobi field equation.  With the initial condition $V(0) = 0$ and $V'(0) = e_1$, it follows from the Jacobi equation that
\begin{align}\label{eqn_apx_03}
\pl_x|_{(x,y,0)} &= u_1 + \CO(\rho_0^2)u_j ~.
\end{align}
Similarly,
$$   \pl_y|_{(x,y,0)} = u_2 + \CO(\rho_0^2)u_j ~.   $$
The Jacobi field equation can be used to find all the higher order coefficients, see \cite[chapter 1]{ref_Cheeger_Ebin}.

\subsubsection{The vector fields $\pl_x$ and $\pl_y$ on $C\times I$}
Fix $((x,y),z)\in C\times I$, and let $\tilde{\gamma}(t,s) = \varphi(x+s,y,tz)$.  The variational field $\tilde{V}(t) = \frac{\pl}{\pl s}|_{s=0}\tilde{\gamma}(t,s)$ is again a Jacobi field.  It follows from the construction that $\tilde{V}(1) = \pl_x|_{(x,y,z)}$.  By (\ref{eqn_apx_03}), the initial value is
\begin{align}\label{eqn_apx_04}
\tilde{V}(0) &= \pl_x|_{(x,y,0)} = u_1 + \CO(\rho_0^2)u_j ~.
\end{align}
By (\ref{eqn_apx_01}), the initial velocity is
\begin{align}\label{eqn_apx_05}\begin{split}
\tilde{V}'(0) &= (\nabla_{\pl_t} \tilde{J}(t))|_{t=0} = (\nabla_{\tilde{J}(0)}\pl_t) = (\nabla_{\pl_x}ze_3)|_{(x,y,0)} \\
&= z(1+N)u_2 + \CO(\rho_0^2)u_j ~.
\end{split}\end{align}
It follows from the Taylor's theorem and the Jacobi field equation that
\begin{align}\label{eqn_apx_06}
\pl_x &= u_1 + z(1+N)u_2 + \CO(\rho^2)u_j
\end{align}
where $\rho = (x^2+y^2+z^2)^\oh$.  Similarly,
$$   \pl_y = u_2 + z(-1+N)u_1 + \CO(\rho^2)u_j ~.   $$

\subsubsection{The contact form}
The expansion of $\pl_x$ and $\pl_y$ can be used to find out the expansion of $a_1(x,y)$ and $a_2(x,y)$ in (\ref{eqn_apx_00}).  The following vector fields are annihilated by $a$:
\begin{align*}
\pl_x - \langle\pl_x,\pl_z\rangle\pl_z &= \pl_x - \big(y(-1+N) + \CO(\rho^2)\big)\pl_z ~, \\
\pl_y - \langle\pl_y,\pl_z\rangle\pl_z &= \pl_y - \big(x(1+N) + \CO(\rho^2)\big)\pl_z ~.
\end{align*}
Thus, $a = \dd z + (y(-1+N) + \CO(\rho_0^2))\dd x + (x(1+N) + \CO(\rho_0^2))\dd y$.

The coefficient of volume element $B(x,y)$ is the determinant of the coefficients of $\{\pl_x,\pl_y,\pl_z\}$ in $\{u_1,u_2,u_3\}$.  By (\ref{eqn_apx_02}) and (\ref{eqn_apx_03}), $B(x,y) = 1 + \CO(\rho_0^2)$.

\subsubsection{Trivialization of $K^{-1}$}
Note that $u_1$ and $u_2$ do not necessarily belong to $\ker(a)$.  To trivialize the bundle $K^{-1}$, perform the Gram--Schmidt process on $\{v,u_1,u_2\}$.  Denote the output by $\{v,e_1,e_2\}$.  A direct computation shows that
\begin{align}\label{eqn_apx_07}\left\{\begin{aligned}
e_1 &= \pl_x - y(-1+N)\pl_z + \CO(\rho^2)\pl_j ~, \\
e_2 &= \pl_y - x(1+N)\pl_z - 2xN\pl_x + \CO(\rho^2)\pl_j ~.
\end{aligned}\right.\end{align}
It is clear that the $e_1$ and $e_2$ coincide with the initial choice at $p$.  The unitary frame $\frac{1}{\sqrt{2}}(e_1-ie_2)$ trivialize the bundle $K^{-1}$ on the adapted chart.

Let $\{\omega^1,\omega^2,\omega^3 = a\}$ be the dual coframe of $\{e_1,e_2,v\}$.  It follows that
\begin{align}\label{eqn_apx_08}\left\{\begin{aligned}
\omega^1 &= \dd x + 2zN\dd y + \CO(\rho^2)\dd x^j ~, \\
\omega^2 &= \dd y + \CO(\rho^2)\dd x^j ~.
\end{aligned}\right.\end{align}
Let $\theta_i^j$ be the Levi-Civita connection in terms of this frame, i.e.\ $\nabla e_i = \sum_j\theta_i^j e_j$.  By \cite[(2.4)]{ref_Ts1}, only $\theta_1^2$ appears in the canonical Dirac operator, and a direct computation shows that
$$   \theta_1^2 = (1+N)\omega^3 + \CO(\rho)\omega^j ~.   $$

\subsubsection{The base connection}
There is a standard technique to write down the local expression of $A_E$ in terms of the (transverse--Reeb) exponential gauge.  It is a variant of the original argument of Uhlenbeck \cite{ref_Uhlenbeck_YM}, and the detail will be omitted.

In the transverse-Reeb exponential gauge, the unitary connection $A_E$ is equal to
\begin{align}\label{eqn_apx_09}
A_E &= (-\oh y F_{12}(p) - z F_{13}(p) + \CO(\rho^2))\omega^1 + (\oh x F_{12}(p) - z F_{23}(p) + \CO(\rho^2))\omega^2
\end{align}
where $F_{A_E}(p) = \sum_{i<j}F_{ij}(p)\,\omega^i\wedge\omega^j$.  Note that there is no $\omega^3$-component in this gauge.

\subsubsection{The Dirac operator}
With the above discussions, the two components of the Dirac operator $\mr{D}_r$ on $\mr{\psi}=(\mr{\alpha},\mr{\beta})$ are
\begin{align}\label{eqn_apx_10}\left\{\begin{aligned}
\pr_1(\mr{D}_r\mr{\psi}) &= \frac{r}{2}\mr{\alpha} + i\pl_z\mr{\alpha} \\
&\quad - 2\pl_\xi\mr{\beta} - i(\bar{\xi} + N\xi)\pl_z\mr{\beta} - 2izN\pl_x\mr{\beta} + \CO(\rho^2)\pl_j\mr{\beta} + \CO(\rho)\mr{\beta} ~,\\
\pr_2(\mr{D}_r\mr{\psi}) &= 2\pl_{\bar{\xi}}\mr{\alpha} - i(\xi + N\bar{\xi})\pl_z\mr{\alpha} - 2izN\pl_x\mr{\alpha} + \CO(\rho^2)\pl_j\mr{\alpha} + \CO(\rho)\mr{\alpha} \\
&\quad - (\frac{r}{2}+1-N)\mr{\beta} - i\pl_z\mr{\beta} + \CO(\rho)\mr{\beta}
\end{aligned}\right.\end{align}
where $\xi$ is the complex coordinate $x+iy$.  This supplies the detail for \S\ref{subsec_adapted_chart} and \S\ref{sec_adapted_Dirac}.

\subsubsection{Change of gauge}
In (\ref{eqn_apx_10}), the $r$-factors appear in the diagonal.  It is also useful to put the $r$-factor in the off-diagonal term.  Consider the following change of gauge:
\begin{align*}
\mb{\alpha} &= \exp(\frac{i}{2}r(z+Nxy))\mr{\alpha} &\text{and}& &\mb{\beta} = \exp(\frac{i}{2}r(z+Nxy))\mr{\beta} ~.
\end{align*}
With respect to this gauge, (\ref{eqn_apx_10}) is transformed into the equation in \S\ref{sec_local_Dr}.


\end{document}